\documentclass[pdflatex,sn-mathphys-num]{sn-jnl}
\usepackage{graphicx}%
\usepackage{multirow}%
\usepackage{amsmath,amssymb,amsfonts}%
\usepackage{amsthm}%
\usepackage{mathrsfs}%
\usepackage[title]{appendix}%
\usepackage{xcolor}%
\usepackage{textcomp}%
\usepackage{manyfoot}%
\usepackage{booktabs}%
\usepackage{algorithm}%
\usepackage{algorithmicx}%
\usepackage{algpseudocode}%
\usepackage{listings}%
\usepackage{indentfirst}
\numberwithin{equation}{section}

\newcommand{\G}{\mathbb{G}}
\newcommand{\dualG}{\widehat{\G}}
\newcommand{\Lp}[2]{L^{#1}(#2)}

\newcommand{\Bop}{\mathcal{B}}

\newcommand{\tens}{\otimes}
\newcommand{\id}{\operatorname{id}}
\newcommand{\Tr}{\operatorname{Tr}}
\newcommand{\Irr}{\operatorname{Irr}}

\newcommand{\range}{\mathcal{R}}
\newcommand{\entropy}{\mathcal{H}}
\newcommand{\cF}{\mathcal{F}}
\newcommand{\cFp}[1]{\mathcal{F}_{#1}}
\newcommand{\conv}{\ast}
\newcommand{\twconvs}{\ast_{\Omega^{*}}}
\newcommand{\twconv}{\ast_{\Omega}}

\theoremstyle{thmstyleone}%
\newtheorem{theorem}{Theorem}[section]%
\newtheorem{proposition}[theorem]{Proposition}%
\newtheorem{lemma}[theorem]{Lemma}
\newtheorem{corollary}[theorem]{Corollary}

\theoremstyle{thmstyletwo}%
\newtheorem{example}[theorem]{Example}%
\newtheorem{remark}[theorem]{Remark}%

\theoremstyle{thmstylethree}%
\newtheorem{definition}[theorem]{Definition}%

\begin{document}

\title[Covariance Principle for STFT on unimodular Kac algebras]{The Time-Frequency Covariance Principle on Unimodular Kac Algebras}

\author[1]{\fnm{Xiao} \sur{Chen}}\email{chenxiao@sdu.edu.cn}
\equalcont{These authors contributed equally to this work.}

\author[2]{\fnm{Rui} \sur{Liu}}\email{ruiliu@nankai.edu.cn}
\equalcont{These authors contributed equally to this work.}

\author*[2]{\fnm{Yuxuan} \sur{Zheng}}\email{yuxuanzheng@mail.nankai.edu.cn}
\equalcont{These authors contributed equally to this work.}

\affil[1]{\orgdiv{School of Mathematics and Statistics}, \orgname{Shandong University}, \city{Weihai}, \postcode{264209}, \state{Shandong}, \country{PR China}}

\affil[2]{\orgdiv{School of Mathematical Sciences and LPMC}, \orgname{Nankai University}, \city{Tianjin}, \postcode{300071}, \country{PR China}}

\abstract{This paper extends the short-time Fourier transform (STFT), a fundamental tool in time-frequency analysis, to the quantum group setting of unimodular Kac algebras. For a unimodular Kac algebra $\G$, we introduce a time-frequency shift operator that combines left translation and modulation operators. Using a window vector in the Hilbert space $\Lp{2}{\G}$, we define the corresponding STFT and establish its essential analytic properties, including a Plancherel theorem, the Moyal identity, an inversion formula, and a fundamental identity. Furthermore, we explore the projective corepresentation structure of the time-frequency shift operator, and prove that its reflected version induces a continuous projective left representation of the dual quantum group of the quantum double. Finally, we derive the covariance principle and several uncertainty principles.}

\keywords{Short-time Fourier transform, Unimodular Kac algebras,
Locally compact quantum groups, Covariance principle, Projective corepresentations}

\maketitle

\section{Introduction}\label{sec1}

The short-time Fourier transform (STFT), also known as the Gabor transform, is a fundamental tool in time-frequency analysis. Its systematic study on locally compact abelian (LCA) groups was advanced by Feichtinger's introduction of modulation spaces \cite{Fei2003} and Gröchenig's extension of Gabor theory \cite{Gro1998}, with later developments in pseudodifferential operators \cite{Gro2007}. For a second countable LCA group $G$ with Haar measure $\mu$ and dual group $\widehat{G}$, the STFT of a function $f \in L^2(G)$ with respect to a window $g \in L^2(G)$ is defined as
\[
\phi_g f(x,\omega) = \int_G f(t)\overline{\omega(t)  g(x^{-1}t) }\, \mathrm{d}\mu(t), \quad x \in G,\; \omega \in \widehat{G}.
\]
Unlike the Fourier transform, which captures only global frequency content, the STFT provides joint time-frequency information \cite{FK1998, Gro1998, Gro2001}. A rich theory has been developed for the STFT on LCA groups, encompassing numerous uncertainty principles \cite{DLP, Ni2023, Ni2024, Gro1998} and related work on its discrete and continuous aspects \cite{Fei2009, Jak2016, Lo2013,WLH2025}.

Recent years have witnessed significant advances in generalizing time-frequency analysis to non-abelian settings, including frame theory and uncertainty principles \cite{Be2022,Gro2021,EnstadJMAA,FK,Ou2018,Ou2019, Sm2022,WLH2024}. A natural extension of this program is to formulate a complete STFT framework for non-abelian groups---one that inherently contains the classical LCA theory and establishes analogous Plancherel, Moyal, and inversion theorems. The present work undertakes this extension by defining the STFT within the category of unimodular Kac algebras. This quantum group framework not only generalizes locally compact groups but also provides the requisite harmonic analysis structure to naturally host time-frequency shift operators and their associated theory.

The primary obstacle in generalizing the STFT to non-abelian groups is the lack of a Pontrjagin-type duality at the level of group structures; the higher-dimensional unitary representations do not themselves form a group. This necessitates a broader category that unifies a locally compact group and its dual. Kac algebras, introduced independently by Vainerman--Kac \cite{Va1973,Va1974} and Enock--Schwartz \cite{Kac}, provide such a framework, generalizing Pontrjagin duality. This theory was later subsumed and extended by the powerful formalism of locally compact quantum groups (LCQGs) developed by Kustermans and Vaes \cite{kustermans2000locally,vaes2001locally,kustermans2003locally}. In this setting, every LCQG $\G$ possesses a dual $\dualG$ satisfying the biduality theorem $\widehat{\dualG} \cong \G$, thereby fully generalizing classical Pontrjagin duality.

Harmonic analysis on quantum groups has flourished since it was founded. For instance, Caspers and Cooney studied the $L^p$-Fourier transform on LCQGs using non-commutative interpolation \cite{caspers2013p,Hausdorff}. Furthermore, Jiang, Liu and Wu established various uncertainty principles for LCQGs, including the Donoho--Stark and Hirschmann--Beckner inequalities, and characterized their minimizers \cite{Liu2017,Liu}. These results form a crucial foundation for our work. Our approach also leverages the theory of projective (co)representations via unitary 2-cocycles on quantum groups, as studied in the context of Galois objects \cite{De phd}.

The paper is organized as follows. 
Section~\ref{sec2}, as a preliminary, is divided into two parts. The fist part reviews the necessary preliminaries on locally compact quantum groups, with a focus on unimodular Kac algebras and their Fourier analysis. 
In the second part of in Section~\ref{sec2}, we recall some concepts and facts related to 2-cocycles and projective (co)representations in the quantum setting, and present the relationship between the continuous projective left regular representation and the twisted convolution. 
In Section~\ref{sec3}, we introduce the fundamental objects of our study: the quantum time-frequency shift operator and the corresponding short-time Fourier transform on a unimodular Kac algebra $\G$. 
We then establish their core analytical properties, namely, a Plancherel theorem, the Moyal identity, and an inversion formula. 
The next two sections are devoted to the structural aspects. 
In Section~\ref{sec4}, we show that our time-frequency shift operator yields a projective corepresentation, and prove that its reflected version induces a projective representation of the dual of the quantum double. 
In Section~\ref{sec5}, we derive the fundamental identity and the covariance principle.
Finally, we prove several classical uncertainty principles in the last Section.

\bigskip 
\section{Preliminaries}\label{sec2}
\medskip

This section reviews the necessary background on locally compact quantum groups, with a focus on unimodular Kac algebras. We establish notations and recall fundamental concepts from non-commutative integration and quantum group theory.

All Hilbert spaces in this paper are complex; their inner products are linear in the first entry and anti-linear in the second, and are denoted by $\langle \cdot|\cdot \rangle$. 
Given a Hilbert space $H$ and vectors $\xi,\eta \in H$, let $\omega_{\xi,\eta}\in \Bop(H)_*$ be the functional defined by $\omega_{\xi,\eta}(x)=\langle x\xi|\eta \rangle$ for $x\in \Bop(H)$. 
We write $\odot$ for the algebraic tensor product. The symbol $\tens$ denotes either the tensor product of Hilbert spaces, the tensor product of operators, or the tensor product of von Neumann algebras (the intended meaning will always be clear from the context). 
We denote by $\chi$ the flip map for tensor products of algebras or of operators, and use $\Sigma$ for the flip map of Hilbert spaces. 
The symbol $\id_X$ stands for a unit operator on a space or algebra $X$, and sometimes we may write $\id$ for short when no confusion arises.

Let $M$ be a von Neumann algebra equipped with a normal, semi-finite, faithful weight $\varphi$ (abbreviated as an \emph{n.s.f. weight}) and set
\[
\mathfrak{N}_{\varphi } = \{ x \in M : \varphi(x^{*}x) < \infty \}, \qquad 
\mathfrak{M}_{\varphi } = \mathfrak{N}_{\varphi}^{*}\mathfrak{N}_{\varphi },
\]
where $\mathfrak{M}_{\varphi }$ is a $*$-subalgebra of $M$. We write $\mathfrak{M}_{\varphi }^{+}$ for the set of positive elements in $\mathfrak{M}_{\varphi }$.

Let $(H_\varphi, \pi_{\varphi},\Lambda_{\varphi})$ be the GNS construction for $\varphi$, where $\Lambda_{\varphi} : \mathfrak{N}_{\varphi} \to H_\varphi$ is the inclusion map. We may assume that $M$ act on $H_{\varphi}.$ Therefore we will omit $\pi_{\varphi}.$ Consequently,
\[
\langle \Lambda_{\varphi}(x) | \Lambda_{\varphi}(y) \rangle = \varphi(y^*x) \quad \text{for } x,y \in \mathfrak{N}_{\varphi},
\]
and 
\[
x \Lambda_{\varphi}(y) = \Lambda_{\varphi}(xy) \quad \text{for } x \in M,\; y \in \mathfrak{N}_{\varphi}.
\]

We denote by $M_{*}$ the predual of $M$, and by $M_{*}^{+}$ the set of the positive linear functionals on $M$. For $\omega \in M_{*}$, the normal functional $\bar{\omega}\in M_*$ is defined by $\bar{\omega}(x)=\overline{\omega(x^{*})}$, $x\in M$. Moreover, $M_*$ carries a natural $M$-bimodule structure given by
\[
(a\omega)(x) = \omega(xa), \qquad (\omega a)(x) = \omega(ax) \quad (a\in M,\;\omega\in M_*,\; x\in M).
\]

\subsection{Locally compact quantum groups}
\medskip

We now recall the definition of a locally compact quantum group in the von Neumann algebraic setting. For further details, we refer to \cite{kustermans2000locally} and \cite{kustermans2003locally}.

\begin{definition}
A \emph{locally compact quantum group} (LCQG) $\G = (M, \Delta_{\G}, \varphi, \psi)$ consists of:
\begin{enumerate}
    \item a von Neumann algebra $M$;
    \item a coproduct $\Delta_{\G}$, i.e., a unital normal $*$-homomorphism from $M$ to $M \tens M$ satisfying the coassociativity condition
    \[
    (\Delta_{\G} \tens \id) \Delta_{\G} = (\id \tens \Delta_{\G}) \Delta_{\G};
    \]
    \item two n.s.f. weights $\varphi$ and $\psi$, called the left and right Haar weights respectively, which are invariant in the sense that
    \begin{align*}
        \varphi\bigl( (\omega \tens \id)(\Delta_{\G}(x)) \bigr) &= \varphi(x) \omega(1) \quad \text{for all } \omega \in M_{*}^{+},\; x \in \mathfrak{M}_{\varphi}^{+},\\
        \psi\bigl( (\id \tens \omega)(\Delta_{\G}(x)) \bigr) &= \psi(x) \omega(1) \quad \text{for all } \omega \in M_{*}^{+},\; x \in \mathfrak{M}_{\varphi}^{+}.
    \end{align*}
\end{enumerate}
Here $1$ is the unit of $M$, and we will use $1_\G$ instead of $1$ in subsequent sections for the sake of further clarification. The Haar weights are unique up to a positive scalar multiple.
\end{definition}

For every LCQG $\G$, there exists a left multiplicative unitary $W^{\G}$ on $H_\varphi \tens H_\varphi$ and a right multiplicative unitary $V^{\G}$ on $H_\psi \tens H_\psi$, defined by
\begin{align} 
W^{\G*} \bigl( \Lambda_{\varphi}(x) \tens \Lambda_{\varphi}(y) \bigr) &= (\Lambda_{\varphi} \tens \Lambda_{\varphi})\bigl( \Delta_{\G}(y) (x \tens 1) \bigr), && x,y \in \mathfrak{N}_{\varphi}, \notag\\
V^{\G} \bigl( \Lambda_{\psi}(x) \tens \Lambda_{\psi}(y) \bigr) &= (\Lambda_{\psi} \tens \Lambda_{\psi})\bigl( \Delta_{\G}(x) (1 \tens y) \bigr), && x,y \in \mathfrak{N}_{\psi}. \label{def:right multiplicative unitary}
\end{align}
These unitaries implement the coproduct via
\[
\Delta_{\G}(x) = W^{\G*} (1 \tens x) W^{\G} = V^{\G} (x \tens 1) V^{\G*}, \qquad x \in M.
\]
Both $W^{\G}$ and $V^{\G}$ satisfy the pentagon equation of Baaj and Skandalis; for $W^{\G}$ it reads
\[
W_{12}^{\G} W_{13}^{\G} W_{23}^{\G} = W_{23}^{\G} W_{12}^{\G},
\]
where $W_{ij}^{\G}$ is the usual leg-numbering notion (see e.g. \cite[Section 1.1]{vaes2001locally}).
The antipode $S_{\G}$ of $\G$ is the unique $\sigma$-strong$^*$ closed linear operator on $M$ satisfying
\[
S_{\G}\bigl( (\id \tens \omega)(W^{\G}) \bigr) = (\id \tens \omega)(W^{\G*}), \qquad \omega \in M_*.
\]

The dual locally compact quantum group $\dualG$ of $\G$ is defined as follows. Its von Neumann algebra is
\[
\widehat{M} = \overline{\{ (\omega \tens \id)(W^{\G}) : \omega \in \Bop(H_{\varphi})_{*} \}}^{w*} .
\]
The comultiplication $\Delta_{\dualG}$ is given by
\[
\Delta_{\dualG}(x) = \Sigma W^{\G} (x \tens 1) W^{\G*} \Sigma, \qquad x \in \widehat{M},
\]
where $\Sigma$ is the flip on $H_{\varphi} \tens H_{\varphi}$. Then $W^{\dualG} = \Sigma W^{\G*} \Sigma$ is the multiplicative unitary for $\dualG$. For $\omega \in M_*$, we use the standard notion $\lambda_{\G}(\omega) = (\omega \tens \id)(W^{\G}).$ 

Consider the set
\[
\mathcal{I}= \{ \omega \in M_* : \exists\, L \ge 0 \text{ such that } | \omega (x^*) | \le L \Vert \Lambda_{\varphi}(x) \Vert \ \text{for all } x \in \mathfrak{N}_{\varphi} \}.
\]
For every $\omega \in \mathcal{I}$, there exists a unique vector $\xi(\omega) \in \Lp{2}{\G}$ such that
\[
\omega (x^*) = \langle \xi(\omega) | \Lambda_{\varphi}(x) \rangle, \qquad x \in \mathfrak{N}_{\varphi}.
\]
The dual left Haar weight $\widehat{\varphi}$ is the unique n.s.f. weight on $\widehat{M}$ determined by the GNS triple $( H_{\varphi}, \iota , \Lambda_{\widehat{\varphi}} )$ such that  $\lambda_{\G} (\mathcal{I})$ is a $\sigma$-strong-$*$/norm core for $\Lambda_{\widehat{\varphi}}$ and $\Lambda_{\widehat{\varphi}} (\lambda_\G(\omega)) = \xi (\omega), \omega \in \mathcal{I}.$ The dual right Haar weight $\widehat{\psi}$ can be defined in a similar way.

The predual $M_*$ becomes a Banach algebra under the convolution product given in Definition~\ref{def:convolution-predual} below and $\lambda_{\G}$ is an injective algebra homomorphism. It is convenient to introduce the dense subspace
\begin{align*}
M^\sharp_* &= \{ \omega \in M_* \mid \exists\, \omega^\sharp  \in M_* : \lambda_{\G}(\omega)^* = \lambda_{\G}(\omega^\sharp  ) \} \\
&= \{ \omega \in M_* \mid \exists\, \omega^\sharp   \in M_*,\; \forall x \in \operatorname{Dom}(S_{\G}) : \overline{\omega}(S_{\G}(x)) = \omega^\sharp  (x) \}.
\end{align*}
The map $\omega \mapsto \omega^\sharp  $ is a well-defined antilinear involution on $M^\sharp_*$, turning it into a $*$-algebra. The same construction yields $\widehat{M}^\sharp_*$ for $\dualG$.

The reduced quantum group $C^{*}$-algebra of $\G$ is
\[
C_{0}(\G) = \overline{ \{ (\id \tens \omega)(W^{\G}) : \omega \in \Bop(H_{\varphi})_{*}\}}^{\Vert \cdot \Vert}.
\]

In this paper, we focus on \emph{unimodular Kac algebras}, i.e., LCQGs $\G = (M, \Delta_{\G}, \varphi, \psi)$ for which $\varphi = \psi$ is tracial. 
By \cite[Proposition 6.1.4]{Kac}, the dual $\dualG$ is again a unimodular Kac algebra. 
For Kac-type quantum groups, the antipode $S_{\G}$ coincides with the unitary antipode $R_{\G}$, and one has $M^\sharp_* = M_*$.

For $1 \le p <\infty$, let $\Lp{p}{\G}$ denote the noncommutative $L^p$-space associated with the tracial Haar weight $\varphi$; it is obtained as the closure of $\mathfrak{M}_{\varphi}$ under the norm $\|x\|_p := \varphi(|x|^p)^{1/p}$ (see \cite{Takesaki} for details). 
In the sequel, we denote $M$ by $\Lp{\infty}{\G}$, unless otherwise stated. 
When $p = 2$, the space $\Lp{2}{\G}$ is a Hilbert space that can be canonically identified with $H_{\varphi}$. The quadruple $\{ M, \Lp{2}{\G}, \Lp{2}{\G}_+, J_{\varphi} \}$ is the standard form of $M$, where the modular conjugation $J_{\varphi}$ satisfies $J_{\varphi} \Lambda_{\varphi}(x) = \Lambda_{\varphi}(x^*)$ for $x \in \mathfrak{N}_{\varphi}$.

The trace $\varphi$ extends uniquely to a positive linear functional on $\Lp{1}{\G}$. For $x \in \Lp{1}{\G}$ define $\varphi_x(y) := \varphi(xy)$ ($y \in M$). Then the map
\[
j : \Lp{1}{\G} \to M_*, \quad x \mapsto \varphi_x
\]
is an isometric isomorphism. By \cite[Proposition 3.1]{Zhn2023}, we have $\Lp{1}{\G}\cap \Lp{2}{\G}=j^{-1}(\mathcal{I}).$
Throughout the paper, unless stated otherwise, we will denote by $\G$ a unimodular Kac algebra with trace $\varphi$, and by $\dualG$ its dual with trace $\widehat{\varphi}.$

\begin{example}[Locally compact unimodular group]\label{ex:lca-group}
Let $G$ be a locally compact unimodular group with a fixed Haar measure $\mu$. Then 
$\G = (\Lp{\infty}{G}, \Delta, \varphi, \varphi)$ is a commutative unimodular Kac algebra, where
\[
\Delta(f)(x,y) = f(xy) \; (x,y \in G), \qquad \varphi(f) = \int_G f(x) \, \mathrm{d}\mu(x).
\]
The multiplicative unitary $W$ acts on $\Lp{2}{G} \tens \Lp{2}{G}$ as $(WF)(x,y) = F(x, x^{-1}y)$, and the antipode is given by $S(f)(x) = f(x^{-1})$ for $f \in \Lp{\infty}{G}$.  

The group von Neumann algebra $\mathcal{L}(G)$ is generated by the left regular representation:
\[
\mathcal{L}(G) = \overline{\{ \lambda_f \mid \lambda_f (g)(t) = \int_G f(x) g(x^{-1}t) \, \mathrm{d}\mu,\; g \in C_c(G,\mu) \}}^{w^*}.
\]
Its coproduct is $\widehat{\Delta}(\lambda_x) = \lambda_x \tens \lambda_x$ ($x \in G$), the antipode is $\widehat{S}(\lambda_f) = \lambda_{\widehat{S}(f)}$ with $\widehat{S}(f)(x) = f(x^{-1})$, and the Haar weight is $\widehat{\varphi}(f) = f(e)$ (where $e$ is the identity of $G$). The cocommutative quantum group $\dualG = (\mathcal{L}(G), \widehat{\Delta}, \widehat{\varphi}, \widehat{\varphi})$ is the dual of $\G$.
\end{example}

\begin{example}[Compact quantum group of Kac type]\label{ex:compact-kac}
A LCQG $\G$ is called a \emph{compact quantum group} (CQG) if $C_0(\G)$ is unital; in this case we write $C(\G) := C_0(\G)$. Every CQG admits a unique Haar state $h$ satisfying
\[
(h \tens \id) \circ \Delta(x) = h(x) \mathbf{1} = (\id \tens h) \circ \Delta(x), \qquad x \in C(\G).
\]

A (non-degenerate, unitary) representation of a compact quantum group $\G$ on a Hilbert space $H_U$ is an invertible (respectively unitary) element $U \in C(\G) \tens \Bop(H_U)$ such that $U_{13} U_{23} = (\Delta \tens \id)(U)$. Denote by $\Irr(\G)$ the set of equivalence classes of irreducible finite-dimensional unitary representations of $\G$ (for CQGs all irreducible representations are finite-dimensional). For each $\pi \in \Irr(\G)$, we fix a representative $u^{(\pi)} \in C(\G) \tens \Bop(H_{\pi})$ and write $n_\pi = \dim H_\pi$.

If the Haar state $h$ is tracial, then $\G$ is a unimodular Kac algebra. In this case, one has the orthogonality relations
\[
h\bigl( u_{ij}^{(\pi)} (u_{lm}^{(\pi')})^* \bigr) = \frac{\delta_{\pi \pi'} \delta_{il} \delta_{mj}}{n_{\pi}}, \qquad 
h\bigl( (u_{ij}^{(\pi)})^* u_{lm}^{(\pi')} \bigr) = \frac{\delta_{\pi \pi'} \delta_{jm} \delta_{li}}{n_{\pi}},
\]
where $\pi, \pi' \in \Irr(\G)$, $1 \le i,j \le n_\pi$, $1 \le l,m \le n_{\pi'}$.  

The dual quantum group $\dualG$ is described by its ``function algebras''
\[
c_0(\dualG) = \bigoplus_{\pi \in \Irr(\G)}^{c_0} \Bop(H_{\pi}), \qquad 
\ell^{\infty}(\dualG) = \bigoplus_{\pi \in \Irr(\G)} \Bop(H_{\pi}),
\]
where the direct sums are taken in the $c_0$- and $\ell^{\infty}$-sense, respectively. The left Haar weight $\widehat{h}$ of $\dualG$ is
\begin{equation} \label{def:left Haar weight of dual CQG}
    \widehat{h}(x) = \sum_{\pi \in \Irr(\G)} n_\pi \Tr(P_\pi x), \qquad x \in \ell^{\infty}(\dualG),
\end{equation}
with $\Tr$ being the usual trace on $\Bop(H_\pi)$ and $P_\pi$ the projection onto $H_\pi$.

Examples of compact quantum groups of Kac type are the free orthogonal quantum groups $O_N^+$ and the free permutation quantum groups $S_N^+$; they are neither commutative nor cocommutative. For more details about compact quantum groups and examples, the readers may refer to \cite{Wo1995,freep, symmetry}.
\end{example}
\subsection{Convolution and projective (co)representation}
\medskip

The correspondence between a locally compact quantum group $\G$ and its dual $\dualG$ extends Pontryagin duality to the quantum setting. In this framework, one can define a Fourier transform on quantum groups.  For unimodular Kac algebras, the definition of the $L^p$-Fourier transform simplifies considerably, compared with the general case treated in \cite{caspers2013p,Hausdorff}.

\begin{definition}[Fourier transform on unimodular Kac algebras]\label{def:fourier}
Let $\G$ be a unimodular Kac algebra. For $1 \le p \le 2$ with $\frac{1}{p} + \frac{1}{q} = 1$, the \emph{Fourier transform} $\cFp{p} : \Lp{p}{\G} \to \Lp{q}{\dualG}$ is the (unique) norm-continuous extension of the map
\[
x \mapsto \lambda_{\G}(\varphi_x), \qquad x \in \mathfrak{M}_{\varphi},
\]
where $\varphi_x(y) = \varphi(xy)$ for $y \in \Lp{\infty}{\G}$.
\end{definition}

On the dual quantum group $\dualG$, one defines analogously the Fourier transform 
$\widehat{\cF}_p : \Lp{p}{\dualG} \to \Lp{q}{\G}$ by 
$\widehat{\cF}_p(x) = \lambda_{\dualG}(\widehat{\varphi}_x)$ for $x \in \mathfrak{M}_{\widehat{\varphi}}$.
According to \cite[Theorem 5.2(1), Corollary 5.5]{caspers2013p}, the map $\widehat{\cF} := \widehat{\cF}_2$ is unique and is the inverse of $\cF:=\cFp{2}.$

For $x \in \Lp{1}{\G}$ we denote $$x^\sharp := j^{-1}(\varphi_x^\sharp  ),$$ where $\varphi_x^\sharp$ is the involution on $M_*$ introduced in Subsection~2.1. 
Then $$(\cFp{1} x)^* = \cFp{1}(x^{\sharp  }).$$
When $x \in \mathfrak{M}_{\varphi}$, one has
\[
\varphi_x^\sharp (y) = \overline{\varphi_x(S_{\G}(y^*))} = \varphi\bigl(S_{\G}(x^*)y\bigr),
\]
so that $$x^\sharp= S_{\G}(x^*) \in \mathfrak{M}_{\varphi}.$$

We now recall the convolution product on quantum groups.

\begin{definition}[Convolution of functionals]\label{def:convolution-predual}
Let $\G = (\Lp{\infty}{\G}, \Delta_{\G}, \varphi, \psi)$ be a locally compact quantum group. For $\omega, \theta \in \Lp{\infty}{\G}_*$, their \emph{convolution} $\omega \conv\theta \in \Lp{\infty}{\G}_*$ is defined by
\[
(\omega\conv \theta)(x) = (\omega \tens \theta)\bigl( \Delta_{\G}(x) \bigr), \qquad x \in \Lp{\infty}{\G}.
\]
\end{definition}

For a unimodular Kac algebra the convolution of elements in $\Lp{1}{\G}$ can be expressed in a more concrete form.

\begin{definition}[Convolution of $L^1$-elements]\label{def:convolution-L1}
Let $\G$ be a unimodular Kac algebra. For any $a, c \in \mathfrak{M}_{\varphi}$, their \emph{convolution product} $a \conv c$ is defined by
\[
a \conv c = j^{-1}(\varphi_a \conv\varphi_c)  \in \mathfrak{M}_{\varphi}.
\]
which is equivalent to saying that $(\varphi_a\conv \varphi_c)(x^*) = \varphi\bigl(x^*(a \conv c)\bigr) = \bigl\langle \Lambda_{\varphi}(a \conv c) \bigm| \Lambda_{\varphi}(x) \bigr\rangle$ for any $x\in\mathfrak{N}_{\varphi}$.
\end{definition}

\begin{remark}\label{rem:convolution-properties}
For $x \in \mathfrak{N}_{\varphi}$ and $a, c \in \mathfrak{M}_{\varphi}$ we have $\varphi_a, \varphi_c, \varphi_a \conv\varphi_c \in \mathcal{I}$. Combining Definitions~\ref{def:fourier} and~\ref{def:convolution-L1} yields
\[
\cF(a \conv c) = \lambda_{\G}(\varphi_a \conv\varphi_c) = \cF(a) \cF(c).
\]
The last equality holds because $\lambda_\G$ is multiplicative. 
Since $\xi(\varphi_a \conv\varphi_c) = \lambda_{\G}(\varphi_a) \xi(\varphi_c)$ (see \cite[Proposition 1.11.5]{vaes2001locally}), it follows that $\Lambda_{\varphi}(a \conv c) = \lambda_{\G}(\varphi_a) \Lambda_{\varphi}(c)$. Consequently, for $a \in \Lp{1}{\G}$ and $c \in \Lp{2}{\G}$,
\[
a \conv c = \lambda_{\G}(\varphi_a)(c) = \cFp{1}(a)(c) \in \Lp{2}{\G}.
\]
Moreover, for any $a,c\in  \mathfrak{M}_{\varphi}$,  one has
\begin{equation}\label{eq:conv-equiv-express}
a \conv c = (\varphi \tens \id)\Bigl( \bigl[ (S_{\G} \tens \id)(\Delta_{\G}(c)) \bigr] (a \tens 1_{\G}) \Bigr).
\end{equation}
(cf, e.g. \cite[Proposition 3.11]{kahng2010fourier}).
\end{remark}

A \emph{unitary $2$-cocycle} on a locally compact quantum group $\G$ is a unitary element $\Omega \in \Lp{\infty}{\G} \tens \Lp{\infty}{\G}$ that satisfies the cocycle identity
\begin{equation}\label{eq:cocycle}
(\Omega \tens 1_{\G})(\Delta_{\G} \tens \id)(\Omega) = (1_{\G} \tens \Omega)(\id \tens \Delta_{\G})(\Omega).
\end{equation}
Given such a cocycle, one can define a new coproduct $\Delta_{\G_{\Omega}}$ on $\Lp{\infty}{\G}$ by
\[
\Delta_{\G_{\Omega}}(x) = \Omega \Delta_{\G}(x) \Omega^*, \qquad x \in \Lp{\infty}{\G}.
\]
The pair $\G_{\Omega} = (\Lp{\infty}{\G}, \Delta_{\G_{\Omega}})$ is again a locally compact quantum group, called the \emph{$\Omega$-twisted} locally compact quantum group, and its multiplicative unitary is denoted by $W^{\G_{\Omega}}$. When $\G$ is a classical LCA group, the twisted quantum group of $\Lp{\infty}{G}$ is isomorphic to the original one. For detailed accounts, we refer to \cite{De phd}.

In the present paper, we are primarily concerned with projective (co)representations associated to a 2-cocycle. By \cite[Proposition 9.1.2]{De phd} and the notation after \cite[Lemma 7.3.1]{De phd}, the following definition, which is a special case of \cite[Definition 10.1.1]{De phd}, suffices for our purposes.

\begin{definition}[$\Omega$-projective corepresentation]\label{def:projective-corep}
Let $\G$ be a locally compact quantum group and $\Omega$ a 2-cocycle on $\G$. A \emph{(unitary) $\Omega$-projective left corepresentation} on a Hilbert space $H$ is a unitary element $u \in \Lp{\infty}{\G} \tens \Bop(H)$ satisfying
\begin{equation}\label{eq:projective-corep}
(\Omega \Delta_{\G} \tens \id)(u) = u_{13} u_{23}.
\end{equation}
\end{definition}

Such a corepresentation $u$ is called \emph{square-integrable} if the associated left coaction
\begin{equation}\label{eq:int-act}
 \alpha : \Bop(H) \to \Lp{\infty}{\G} \tens \Bop(H), \quad x \mapsto u^*(1_{\G} \tens x) u   
\end{equation}
is integrable, i.e., if $\mathfrak{M}_{\varphi \tens \id} \cap \alpha(\Bop(H))$ is $\sigma$-weakly dense in $\alpha(\Bop(H))$, see \cite[Definition 10.2.2]{De phd}.

An \emph{intertwiner} between two $\Omega$-projective left corepresentations $u_1$ and $u_2$ on Hilbert spaces $H_1$ and $H_2$, respectively, is an operator $x : H_2 \to H_1$ such that $u_1(1_{\G} \tens x) = (1_{\G} \tens x) u_2$. A corepresentation is \emph{irreducible} if its self-intertwiners are precisely the scalar multiples of the identity \cite[Page 303]{De phd}.

According to \cite[Proposition 2.4]{Tuset deformation},  for any unitary 2-cocycle $\Omega$ on $\G$,  the element $\Omega^*$ is a unitary 2-cocycle on $\G_{\Omega}$.  
By \cite[Equation 2.1 and 2.5]{Tuset deformation} and \cite[Section 10.1]{De phd}, we can see that, the operator $W^{\G} \Omega^*$, called the \emph{unitary $\Omega^{*}$-projective left regular corepresentation} of $\G_{\Omega}$, is a unitary $\Omega^{*}$-projective left corepresentation of $\G_{\Omega}$ on $\Lp{2}{\G}$, while $W^{\G_{\Omega}} \Omega$, called the \emph{unitary $\Omega$-projective left regular corepresentation} of $\G_{\Omega}$, is a unitary $\Omega$-projective left corepresentation of $\G$ on $\Lp{2}{\G}$.

As said in \cite[Lemma 7.2.6 and Section 11.2]{De phd} and \cite[Section 2.1]{Tuset deformation}, we know that the norm closure of the space of operators $(\omega \tens \id)(W^{\G} \Omega^*)$, $\omega \in \Bop(\Lp{2}{\G})_*$, is called \emph{reduced twisted group $C^{*}$-algebra} denoted by $C_r^{*}(\G; \Omega)$, and the von Neumann algebra $C_r^{*}(\G; \Omega)''\subset\Bop(\Lp{2}{\G})$ is denoted by $W^{*}(\G; \Omega)$ which is exactly the Galois object, namely the twisting by 2-cocycle in \cite[Section 9.1]{De phd}. 
Similarly, for the conjugate 2-cocycle $\Omega^{*}$, by using another unitary $W^{\G_{\Omega}} \Omega$, we also obtain the reduced twisted group $C^{*}$-algebra $C_r^{*}(\G_{\Omega}; \Omega^{*})$ and the Galois object $W^{*}(\G_{\Omega}; \Omega^{*})$ of $\G_{\Omega}$.

\begin{remark}\label{rem:sqr-int-galoismap}
Since  $W^{*}(\G; \Omega)$ is the Galois object of $\G$ (i.e., the twisting by 2-cocycle $\Omega$),  the unitary $W^{\G} \Omega^*$ is the Galois map, the associated action defined as in \eqref{eq:int-act} is semidual and hence integrable. 
Then $W^{\G} \Omega^*$, as the unitary $\Omega^{*}$-projective left corepresentation of $\G_{\Omega}$, is square-integrable.
Similarly, the unitary $\Omega$-projective left corepresentation $W^{\G_{\Omega}} \Omega$ of $\G$ is also square-integrable. 
For more details, the reader may refer to \cite[Section 9.1, Example 7.1.2 and Section 7.3]{De phd} and \cite[proposition 2.7.12]{vaes2001locally}.
\end{remark}

We define the \emph{$\Omega$-twisted convolution} of functionals $\omega, \mu \in \Lp{\infty}{\G}_*$ by
\[
(\omega \twconv \mu)(x) := (\omega \tens \mu)\bigl( \Delta_{\G}(x) \Omega^* \bigr), \qquad x \in \Lp{\infty}{\G}.
\]
It is straightforward to verify that $\omega \twconv \mu$ again belongs to $\Lp{\infty}{\G}_*$. 
As explained in \cite[Section 11.3]{De phd}, there exists a compatible involution $\sharp_{\Omega}$ on a suitable subspace of $\Lp{\infty}{\G}_*$; the universal $C^{*}$-algebraic envelope of $\Lp{\infty}{\G}_*^{\sharp_{\Omega}}$ under the twisted convolution and this involution is called the \emph{universal twisted group $C^{*}$-algebra} of $\G$ and denoted by $C_u^{*}(\G; \Omega)$. 
For the twisted quantum group $\G_{\Omega}$, replacing $\Omega$ by $\Omega^{*}$, we can directly obtain the $\Omega^*$-twisted convolution on $\Lp{\infty}{\G_{\Omega}}_* = \Lp{\infty}{\G}_*$ as follows:
 $$(\omega \twconvs \mu)(x) = (\omega \tens \mu)( \Delta_{\G_{\Omega}}(x) (\Omega^{*})^{*})= (\omega \tens \mu)(\Omega\Delta_{\G}(x)), $$
 which induces the universal twisted group $C^{*}$-algebra $C_u^{*}(\G_{\Omega}; \Omega^{*})$ of $\G_{\Omega}$.

As a special case of \cite[Definition 10.2.1]{De phd}, the notion of projective representation on the dual side is introduced as follows.

\begin{definition}[$\Omega$-projective representation of the dual]\label{def:projective-rep-dual}
Let $\G$ be a locally compact quantum group and $\Omega$ a 2-cocycle on $\G$. A \emph{continuous $\Omega$-projective left representation} of $\dualG$ is a non-degenerate $*$-representation of the universal twisted group $C^{*}$-algebra $C_u^{*}(\G_{\Omega}; \Omega^{*})$ on some Hilbert space.
\end{definition}

\begin{remark}\label{rem:corep-rep-correspondence}
By \cite[Proposition 10.2.3(1) and Proposition 11.3.8]{De phd}, there is a one-to-one correspondence between (irreducible) $\Omega$-projective corepresentations of $\G$ and (irreducible) continuous  $\Omega$-projective representations of $\dualG$. More precisely, if $v$ is a $\Omega$-projective left corepresentation of $\G$ on some Hilbert space $H$, then the corresponding continuous $\Omega$-projective left representations of $\dualG$ is given by
\[
\pi_v : \Lp{\infty}{\G}_* \to \Bop(H), \quad \omega \mapsto (\omega \tens \id)(v),
\]
satisifying
\begin{equation}\label {prop:proj-corep-homomorphism}
\pi_v(\omega) \pi_v(\mu) = \pi_v(\omega \twconvs \mu), \qquad \omega, \mu \in \Lp{\infty}{\G}_*.
\end{equation}

\end{remark}

Set
\[
\lambda^{\Omega}_{\G}(\omega) := (\omega \tens \id)(W^{\G} \Omega^*), \quad \omega \in \Bop(\Lp{2}{\G})_*,
\]
which is a continuous $\Omega^{*}$-projective left representation of $\widehat{\G_{\Omega}}$.
We call $\lambda^{\Omega}_{\G}$ the \emph{continuous $\Omega^{*}$-projective left regular representation} of $\widehat{\G_{\Omega}}$.

\begin{remark}
It is worth noticing that, in case $\G = \Lp{\infty}{G}$ for an ordinary locally compact group $G$ with an ordinary 2-cocycle $\Omega$, the $\Omega$-twisted convolution $\twconv$ is in fact the ordinary twisted convolution by the complex conjugate 2-cocycle $\bar\Omega$, while $\twconvs$ is the one by the original 2-cocycle $\Omega$. Hence, in the classical case, the twisted group algebra $C_r^{*}(\G; \Omega)$, $C_u^{*}(\G; \Omega)$ and $W^{*}(\G; \Omega)$ (resp. $C_r^{*}(\G_{\Omega}; \Omega^{*})$, $C_u^{*}(\G_{\Omega}; \Omega^{*})$ and $W^{*}(\G_{\Omega}; \Omega^{*})$) are the ordinary twisted group algebra $C_r^{*}(G)_{\bar\Omega}$, $C_u^{*}(G)_{\bar\Omega}$ and $L(G)_{\bar\Omega}$ (resp. $C_r^{*}(G)_{\Omega}$, $C_u^{*}(G)_{\Omega}$ and $L(G)_{\Omega}$), respectively.
Here the corepresentation $W^{\G} \Omega^*$ is just the ordinary $\bar\Omega$-projective left regular representation $\lambda_{\bar\Omega}$ of $G$ on $\Lp{2}{G}$, and then $\lambda^{\Omega}_{\G}$ is the classical $*$-lift of $\lambda_{\bar\Omega}$  to $C_u^{*}(G)_{\bar\Omega}$,
but $W^{\G_{\Omega}} \Omega$ exactly goes back to the ordinary $\Omega$-projective left regular representation $\lambda_\Omega$ of $G$ on $\Lp{2}{G}$, and $(\omega \tens \id)(W^{\G_\Omega} \Omega)$ is the $*$-lift of $\lambda_\Omega$  to $C_u^{*}(G)_{\Omega}$.
However, since it is well-known that $C_r^{*}(\G; \Omega)$ is $*$-anti-isomorphic to $C_r^{*}(\G_{\Omega}; \Omega^{*})$ (see \cite[Proposition 2.4]{Tuset deformation}), the above difference has no effect on our subsequent discussions.
\end{remark}

The next proposition relates this continuous projective left regular representation to twisted convolution.

\begin{proposition}\label{prop:twisted-convolution-regular}
Let $\G$ be a locally compact quantum group and $\Omega$ a $2$-cocycle on $\G$. For every $\omega \in \Lp{\infty}{\G}_*$ and $\mu \in \mathcal{I}$, the twisted convolution $\omega \twconv \mu$ belongs to $\mathcal{I}$ and
\[
\xi(\omega \twconv \mu) = \lambda^{\Omega}_{\G}(\omega) \, \xi(\mu).
\]
\end{proposition}

\begin{proof}
For $x \in \mathfrak{N}_{\varphi}$, it follows from \cite[Proposition 7.3.4]{De phd} and \cite[Proposition 9.1.2]{De phd} that $(\overline{\omega} \tens \id)(\Omega \Delta_{\G}(x)) \in \mathfrak{N}_{\varphi}$ and
\[
(\overline{\omega} \tens \id)(\Omega W^{\G*}) \Lambda_{\varphi}(x) = \Lambda_{\varphi}\!\bigl((\overline{\omega} \tens \id) \Omega \Delta_{\G}(x) \bigr).
\]
Consequently, we have
\begin{align*}
(\omega \twconv \mu)(x^*) 
&= (\omega \tens \mu)\bigl( \Delta_{\G}(x^*) \Omega^* \bigr) 
= \mu\!\bigl( (\omega \tens \id) \Delta_{\G}(x^*) \Omega^* \bigr) \\
&= \mu\!\Bigl( \bigl( (\overline{\omega} \tens \id) \Omega \Delta_{\G}(x) \bigr)^* \Bigr)
= \bigl\langle \xi(\mu) \bigm| \Lambda_{\varphi}\!\bigl( (\overline{\omega} \tens \id) \Omega \Delta_{\G}(x) \bigr) \bigr\rangle \\
&= \bigl\langle \xi(\mu) \bigm| (\overline{\omega} \tens \id)(\Omega W^{\G*}) \Lambda_{\varphi}(x) \bigr\rangle \\
&= \bigl\langle (\omega \tens \id)(W^{\G} \Omega^*) \xi(\mu) \bigm| \Lambda_{\varphi}(x) \bigr\rangle.
\end{align*}
This shows that $\omega \twconv \mu \in \mathcal{I}$ and $\xi(\omega \twconv \mu) = \lambda^{\Omega}_{\G}(\omega) \xi(\mu)$.
\end{proof}

\bigskip
\section{Short-time Fourier transform}\label{sec3}
\medskip

This section introduces the fundamental objects of time-frequency analysis on unimodular Kac algebras. We define translation and modulation operators, which combine to yield a quantum analogue of the time-frequency shift operator. Using these ingredients, we define the short-time Fourier transform (STFT) on a unimodular Kac algebra $\G$ and establish its basic analytic properties: a Plancherel theorem, the Moyal identity, and an inversion formula. At the end of the section, we verify that all definitions reduce to the classical ones when $\G$ is a locally compact abelian group and the non-abelian case when $\G$ is a unimodular locally compact groups of type~I.

\begin{definition}[Translation operator]\label{def:translation}
For $a \in \Lp{1}{\G}$ and $x \in \Lp{2}{\G}$, the \emph{left translation operator} $T_a^{\G}$ on $\Lp{2}{\G}$ is defined by
\[
T_a^{\G} = \lambda_{\G}(\varphi_a) = (\varphi_a \tens \id)(W^{\G}) \in \Lp{\infty}{\dualG} \subseteq \Bop(\Lp{2}{\G}),
\]
and its action is given by
\[
T_a^{\G} x = \lambda_{\G}(\varphi_a)(x) = a \conv x.
\]
\end{definition}

\begin{definition}[Modulation operator]\label{def:modulation}
For $b \in \Lp{1}{\dualG}$ and $x \in \Lp{2}{\G}$, the \emph{left modulation operator} $E_b^{\G}$ on $\Lp{2}{\G}$ is defined by
\[
E_b^{\G} = \lambda_{\dualG}(\widehat{\varphi}_b) = (\id \tens \widehat{\varphi}_b)(W^{\G*}) \in \Lp{\infty}{\G} \subseteq \Bop(\Lp{2}{\G}),
\]
and its action is given by
\[
E_b^{\G} x = \lambda_{\dualG}(\widehat{\varphi}_b) x = \widehat{\cF}_1(b) x.
\]
\end{definition}

\begin{remark}\label{rem:unitary-equivalence}
There is a natural unitary equivalence between the modulation operator and the dual translation operator. For $b \in \Lp{1}{\dualG}$ and $x \in \Lp{2}{\dualG}$, using \cite[Theorem 6.4(4)]{caspers2013p} we obtain
\[
\cF E_b^{\G} \widehat{\cF}(x) = \cF\bigl( \widehat{\cF}_1(b) \widehat{\cF}(x) \bigr) 
= \cF \cF^{-1}(b \conv x) = b \conv x = T_b^{\dualG} x.
\]
Thus $\cF E_b^{\G} = T_b^{\dualG} \cF$. Similarly, for every $a \in \Lp{1}{\G}$ we have $E_a^{\dualG} \cF = \cF T_a^{\G}$.
\end{remark}

For $\nu \in \Lp{1}{\G} \tens \Lp{1}{\dualG}$, the map $(\varphi \tens \widehat{\varphi})_\nu \odot \id$ extends uniquely to a norm-continuous normal linear map
\[
(\varphi \tens \widehat{\varphi})_\nu \tens \id : \Lp{\infty}{\G} \tens \Lp{\infty}{\dualG} \tens \Bop(\Lp{2}{\G}) \to \Bop(\Lp{2}{\G})
\]
with $\| (\varphi \tens \widehat{\varphi})_\nu \tens \id \| = \| \nu \|_1$.  
Using the leg-numbering notation we write $(\varphi \tens \widehat{\varphi})_{\nu 13}$ for the map
$
\bigl( (\varphi \tens \widehat{\varphi})_\nu \tens \id \bigr) \circ (\id \tens \chi)$,
where $\chi$ is the flip from $\Bop(\Lp{2}{\G}) \tens \Lp{\infty}{\dualG}$ to $\Lp{\infty}{\dualG} \tens \Bop(\Lp{2}{\G})$; note that $\| (\varphi \tens \widehat{\varphi})_{\nu 13} \| = \| \nu \|_1$.

We define a map $\pi_{\G} : \Lp{1}{\G} \tens \Lp{1}{\dualG} \to \Bop(\Lp{2}{\G})$ by
\[
\pi_{\G}(\nu) = \bigl( (\varphi \tens \widehat{\varphi})_\nu \tens \id \bigr)(W_{23}^{\dualG} W_{13}^{\G})
= (\varphi \tens \widehat{\varphi})_{\nu 13}(W^{\G*}_{23} W^{\G}_{12}), \qquad \nu \in \Lp{1}{\G} \tens \Lp{1}{\dualG}.
\]
Because $\| \pi_{\G}(\nu) \| \le \| \nu \|_1  \| W_{23}^{\dualG} W_{13}^{\G} \| \le \| \nu \|_1$, the map $\pi_{\G}$ is contractive.  
For $a \in \Lp{1}{\G}$ and $b \in \Lp{1}{\dualG}$ with $\nu = a \tens b$, a straightforward calculation gives
\[
\pi_{\G}(\nu) = (\varphi_a \tens \id \tens \widehat{\varphi}_b)(W^{\G*}_{23} W^{\G}_{12})
= (\id \tens \widehat{\varphi}_b)(W^{\G*}) (\varphi_a \tens \id)(W^{\G}) = E^{\G}_b T^{\G}_a.
\]

\begin{definition}[Time-frequency shift operator]\label{def:time-freq-shift}
For $\nu \in \Lp{1}{\G} \tens \Lp{1}{\dualG}$, the operator
\[
\pi_{\G}(\nu) := \bigl( (\varphi \tens \widehat{\varphi})_\nu \tens \id \bigr)(W_{23}^{\dualG} W_{13}^{\G})
= (\varphi \tens \widehat{\varphi})_{\nu 13}(W^{\G*}_{23} W^{\G}_{12})
\]
is called the \emph{time-frequency shift operator}. Its \emph{reflected} version is defined by
\[
\pi^{re}_{\G}(\nu) := \bigl( (\varphi \tens \widehat{\varphi})_\nu \tens \id \bigr)(W_{13}^{\G} W_{23}^{\dualG})
= (\varphi \tens \widehat{\varphi})_{\nu 13}(W^{\G}_{12} W^{\G*}_{23}).
\]
\end{definition}

We now introduce the short-time Fourier transform on a unimodular Kac algebra.

\begin{definition}[Short-time Fourier transform]\label{def:STFT}
Let $x \in \Lp{2}{\G}$ and let $y \in \Lp{2}{\G}$ be a window vector. The \emph{short-time Fourier transform} of $x$ with respect to $y$ is the element of $\Lp{2}{\G} \tens \Lp{2}{\dualG}$ defined by
\[
\phi_y^{\G} x := (\id \tens \cF)((J_{\widehat{\varphi}} \tens J_{\varphi}) V^{\G} (y \tens J_{\varphi} x)).
\]
\end{definition}

The next proposition exhibits the relation between the STFT and the time-frequency shift operator, and contains a Plancherel-type identity.

\begin{theorem}[Plancherel theorem]\label{prop:STFT-basic}
Let $x \in \Lp{2}{\G}$ and let $y \in \Lp{2}{\G}$ be a window vector. Then
\[
\| \phi_y^{\G} x \|_2 = \| x \|_2 \| y \|_2.
\]
Moreover, for every $a \in \Lp{1}{\G} \cap \Lp{2}{\G}$ and $b \in \Lp{1}{\dualG} \cap \Lp{2}{\dualG}$, one has
\[
\bigl\langle \phi_y^{\G} x \bigm| a \tens b \bigr\rangle
= \bigl\langle x \bigm| E^{\G}_b \circ T^{\G}_a(y) \bigr\rangle
= \bigl\langle x \bigm| \pi_{\G}(a \tens b)(y) \bigr\rangle.
\]
\end{theorem}

\begin{proof}
First take $x \in \mathfrak{N}_{\varphi}$, $y, a \in \mathfrak{M}_{\varphi}$ and $b \in \mathfrak{M}_{\widehat{\varphi}}$. 
Since $\G$ is Kac type and $\varphi$ is tracial, using \eqref{def:right multiplicative unitary}, \eqref{eq:conv-equiv-express} and \cite[Proposition 1.13.14]{vaes2001locally}, we compute
\begin{align*}
\bigl\langle \Lambda_{\varphi}(x) \bigm| E^{\G}_b \circ T^{\G}_a \Lambda_{\varphi}(y) \bigr\rangle
&= \bigl\langle \Lambda_{\varphi}(x) \bigm| \Lambda_{\varphi}((\cF^{-1} b)(a \conv y)) \bigr\rangle \\
&= \bigl\langle \Lambda_{\varphi}(x) \bigm| \Lambda_{\varphi}\bigl( (\cF^{-1} b) (\varphi \tens \id)[(S_{\G} \tens \id)\Delta_{\G}(y)](a \tens 1_{\G}) \bigr) \bigr\rangle \\
&= (\varphi \tens \varphi)\bigl( (a^* \tens 1_{\G}) [(S_{\G} \tens \id)\Delta_{\G}(y^*)] (1 \tens (\cF^{-1} b)^* x) \bigr) \\
&= (\varphi \tens \varphi)\bigl( (a^* \tens (\cF^{-1} b)^* x) [(S_{\G} \tens \id)\Delta_{\G}(y^*)] \bigr) \\
&= \bigl\langle (\Lambda_{\varphi} \tens \Lambda_{\varphi})\bigl( (1_{\G} \tens x) (S_{\G} \tens \id)\Delta_{\G}(y^*) \bigr) \bigm| \Lambda_{\varphi}(a) \tens \Lambda_{\varphi}(\cF^{-1} b) \bigr\rangle \\
&= \bigl\langle (J_{\widehat{\varphi}} \tens J_{\varphi}) (\Lambda_{\varphi} \tens \Lambda_{\varphi})\bigl( \Delta_{\G}(y) (1_{\G} \tens x^*) \bigr) \bigm| \Lambda_{\varphi}(a) \tens \Lambda_{\varphi}(\cF^{-1} b) \bigr\rangle \\
&= \bigl\langle (J_{\widehat{\varphi}} \tens J_{\varphi}) V^{\G}(\Lambda_{\varphi}(y) \tens \Lambda_{\varphi}(x^*)) \bigm| \Lambda_{\varphi}(a) \tens \Lambda_{\varphi}(\cF^{-1} b) \bigr\rangle \\
&= \bigl\langle (\id \tens \cF)(J_{\widehat{\varphi}} \tens J_{\varphi}) V^{\G} (\Lambda_{\varphi}(y) \tens J_{\varphi} \Lambda_{\varphi}(x)) \bigm| \Lambda_{\varphi}(a) \tens \Lambda_{\widehat{\varphi}}(b) \bigr\rangle \\
&= \bigl\langle \phi_{\Lambda_{\varphi}(y)}^{\G} \Lambda_{\varphi}(x) \bigm| \Lambda_{\varphi}(a) \tens \Lambda_{\widehat{\varphi}}(b) \bigr\rangle.
\end{align*}
Because $\mathfrak{M}_{\varphi}$ is a core for $\Lambda_{\varphi}$, the equality
\[
\bigl\langle \phi_y^{\G} x \bigm| a \tens b \bigr\rangle = \bigl\langle x \bigm| E^{\G}_b \circ T^{\G}_a(y) \bigr\rangle = \bigl\langle x \bigm| \pi_{\G}(a \tens b)(y) \bigr\rangle
\]
holds for all $x, y \in \Lp{2}{\G},a\in \Lp{1}{\G}\cap \Lp{2}{\G}$ and $b \in \Lp{1}{\dualG} \cap \Lp{2}{\dualG} .$  
The operator $(\id \tens \cF)(J_{\widehat{\varphi}} \tens J_{\varphi}) V^{\G} (\id \tens J_{\varphi})$ is an isometry from $\Lp{2}{\G} \tens \Lp{2}{\G}$ onto itself; consequently
\[
\| \phi_y^{\G} x \|_2 = \| x \|_2 \| y \|_2,
\]
which completes the proof.
\end{proof}
\begin{remark}\label{rmk:another def of STFT}
Actually, for $x, y \in \mathfrak{N}_{\varphi}$ the elements $1_{\G} \tens x$ and $(S_{\G} \tens \id)(\Delta_{\G}(y^{*}))$ belong to $\mathfrak{N}_{\id \tens \varphi}$ by \cite[Proposition~1.T1.3]{vaes2001locally}; consequently
$(1_{\G} \tens x) (S_{\G} \tens \id)(\Delta_{\G}(y^{*})) \in \mathfrak{M}_{\id \tens \varphi}$, and from the proof of Theorem~\ref{prop:STFT-basic} one has
\begin{equation}\label{eq:STFT-algebraic}
\phi^{\G}_{y} x = (\id \tens \cF)\bigl( (1_{\G} \tens x) (S_{\G} \tens \id) \Delta(y^{*}) \bigr) \in \mathfrak{N}_{\varphi \tens \widehat{\varphi}}.
\end{equation}
\end{remark}

\begin{corollary}[Moyal identity]\label{cor:Moyal}
For any $x_i, y_i \in \Lp{2}{\G}\;(i=1,2)$, one has
\[
\bigl\langle \phi_{y_1}^{\G} x_1 \bigm| \phi_{y_2}^{\G} x_2 \bigr\rangle
= \langle x_1 \mid x_2 \rangle \; \langle y_2 \mid y_1 \rangle.
\]
\end{corollary}

\begin{proof}
Since $J_{\varphi}$ and $J_{\widehat{\varphi}}$ are anti-linear isometries, the computation in the previous proposition gives
\begin{align*}
\bigl\langle \phi_{y_1}^{\G} x_1 \bigm| \phi_{y_2}^{\G} x_2 \bigr\rangle
&= \bigl\langle (J_{\widehat{\varphi}} \tens J_{\varphi}) V^{\G} (y_1 \tens J_{\varphi} x_1) \bigm| (J_{\widehat{\varphi}} \tens J_{\varphi}) V^{\G} (y_2 \tens J_{\varphi} x_2) \bigr\rangle \\
&= \bigl\langle V^{\G} (y_2 \tens J_{\varphi} x_2) \bigm| V^{\G} (y_1 \tens J_{\varphi} x_1) \bigr\rangle \\
&= \bigl\langle J_{\varphi} x_2 \tens y_2 \bigm| J_{\varphi} x_1 \tens y_1 \bigr\rangle 
= \langle x_1 \mid x_2 \rangle \, \langle y_2 \mid y_1 \rangle.
\end{align*}
\end{proof}

When $\G$ is a compact quantum group (CQG) of Kac type, the Fourier transform and its inverse take a simpler form because the Haar weight is tracial state. %For such quantum groups the following Fourier inversion formula is known (see \cite[Proposition 3.2]{wang2016L}).

%We also need a partial inverse Fourier transform on $\Lp{\infty}{\G} \tens \ell^{\infty}(\dualG)$.

\begin{theorem}[Inversion formula for the STFT on a CQG of Kac type]\label{thm:inversion-STFT-CQG}
Let $\G$ be a compact quantum group of Kac type with tracial state $h$. Choose window vectors $y, \widetilde{y} \in \Lp{2}{\G}$ with $\langle \widetilde{y} \mid y \rangle \neq 0$. Then for every $x \in \Lp{2}{\G}$,
\begin{align*}
     x = \sum_{\pi \in \Irr(\G)} \frac{n_{\pi}}{h(y^*\widetilde{y})} 
      (h \tens \id \tens \Tr) \Bigl( 
      \bigl( ( \id \tens \id \tens P_{\pi})((\phi^{\G}_y x)_{13}) \bigr) 
        \bigl(1_{\G} \tens u^{(\pi)}\bigr)^* \bigl( (S_{\G} \tens \id)(\Delta_{\G}(\widetilde{y})) \tens \id_{H_{\pi}} \bigr) \Bigr).
\end{align*}
\end{theorem}
\begin{proof}
Firstly, we give an expression of partial inverse Fourier transform , i.e., $\id \otimes \cF^{-1}$. Let $t \in \Lp{\infty}{\G} \tens \ell^{\infty}(\dualG)$, using (\ref{def:left Haar weight of dual CQG}), we obtain
\begin{align}
(\id \tens \cF^{-1})(t)
&= ( \id \tens \id \tens \widehat{h})\bigl( t_{13} ( 1_{\G} \tens (\oplus_{\pi \in \Irr(\G)} u^{(\pi)})^*) \bigr) \notag\\
&= \sum_{\pi \in \Irr(\G)} n_{\pi} \, (\id \tens \id \tens \Tr)
      \Bigl( \bigl( (\id \tens \id \tens P_{\pi})(t_{13})\bigr) \bigl(1_{\G} \tens u^{(\pi)}\bigr)^* \Bigr). \label{def:partial-inverse-FT}
\end{align}
Next,
for $x, y, \widetilde{y}, z \in \Lp{2}{\G}$, using Corollary~\ref{cor:Moyal} and (\ref{eq:STFT-algebraic}),  we have
\begin{align*}
\langle x \mid z \rangle 
&= \frac{1}{h(y^*\widetilde{y})} \bigl\langle \phi^{\G}_y x \bigm| \phi^{\G}_{\widetilde{y}} z \bigr\rangle
 = \frac{1}{h(y^*\widetilde{y})} \bigl\langle (\id \tens \cF^{-1}) (\phi^{\G}_y x) \bigm| (\id \tens \cF^{-1}) (\phi^{\G}_{\widetilde{y}} z) \bigr\rangle \\
&= \frac{1}{h(y^*\widetilde{y})} (h \tens h)\Bigl( \bigl( (1_{\G} \tens z) (S_{\G} \tens \id)(\Delta_{\G}(\widetilde{y}^*)) \bigr)^* \bigl((\id \tens \cF^{-1}) (\phi^{\G}_y x) \bigr) \Bigr) \\
&= \frac{1}{h(y^*\widetilde{y})} (h \tens h)\Bigl( \bigl(1_{\G} \tens z^*\bigr) \bigl((\id \tens \cF^{-1})(\phi^{\G}_y x)\bigr) \bigl( (S_{\G} \tens \id)(\Delta_{\G}(\widetilde{y})) \bigr) \Bigr) \\
&= \Bigl\langle \frac{1}{h(y^*\widetilde{y})} (h \tens \id) \Bigl(\bigl( (\id \tens \cF^{-1})(\phi^{\G}_y x) \bigr) \bigl( (S_{\G} \tens \id)(\Delta_{\G}(\widetilde{y})) \bigr) \Bigr) \Bigm| z \Bigr\rangle .
\end{align*}
Hence, we have
\[
x = \frac{1}{h(y^*\widetilde{y})} (h \tens \id) \Bigl(\bigl( (\id \tens \cF^{-1})(\phi^{\G}_y x) \bigr)\bigl( (S_{\G} \tens \id)(\Delta_{\G}(\widetilde{y})) \bigr) \Bigr) .
\]
Finally, substituting the expression for $(\id \tens \cF^{-1})$ from (\ref{def:partial-inverse-FT}) into the above formula yields the stated formula.
\end{proof}

%We verify that our definitions reproduce the classical short-time Fourier transform when the underlying quantum group is a locally compact abelian group.

\begin{example}[Locally compact abelian groups]\label{ex:lcag}
When $\G$ is a locally compact abelian group $G$ with a fixed Haar measure $\mathrm{d}x$, the dual group $\widehat\G$ is the group $\widehat{G}$ of one dimensional unitary representations with a Haar measure $\mathrm{d}\gamma$. For $x \in G $ and $\gamma \in \widehat{G}$, the translation operators $T_x$, the modulation operator $E_\gamma$ on $\Lp{2}{G}$ and the corresponding time-frequency shift $\pi(x,\gamma)$ are defined  by 
\begin{equation}\label{def:time-frequency shift}
    T_xf(y)=f(x^{-1}y), \quad  E_\gamma f(y)=\gamma(y)f(y),\quad  \pi(x,\gamma)f(y)=E_\gamma T_xf(y),
\end{equation}
where $y\in G.$ 
The left regular representation  of $G$ is a representation $\lambda$ of $G$ on $\Lp{2}{G}$ given by $\lambda_x=T_x,x\in G.$
The group representation $\pi : G \to \Bop(\Lp{2}{G})$ is lifted onto the group algebra $\Lp{1}{G}$ via
\[
\widetilde{\pi} : \Lp{1}{G} \to \Bop(\Lp{2}{G}), \qquad
\widetilde{\pi}(f) = \int_{G} f(x) \pi(x) \, \mathrm{d}x,
\]
so that $\widetilde{\pi}(f)(\xi) = \int_{G} f(x) \pi(x) \xi \, \mathrm{d}x$ for $\xi \in \Lp{2}{G}$.  

Therefore, the translation operator $T_x$ ($x \in G$) is lifted to be
\[
\widetilde{T} : \Lp{1}{G} \to \Bop(\Lp{2}{G}), \quad 
\widetilde{T}_g(f)(y) = \int_{G} g(x) T_x(f)(y) \, \mathrm{d}x = (g \conv f)(y),
\]
 and the modulation operator $E_\gamma$ ($\gamma \in \widehat{G}$) is lifted to
\[
\widetilde{E} : \Lp{1}{\widehat{G}} \to \Bop(\Lp{2}{G}), \quad
\widetilde{E}_h(f)(y) = \int_{\widehat{G}} h(\gamma) E_\gamma(f)(y) \, \mathrm{d}\gamma 
= \mathcal{M}_{F^{-1}(h)}(f)(y),
\]
where $F$ is the classical Fourier transform, and $\mathcal{M}_{F^{-1}(h)}$ denotes the multiplication operator by $F^{-1}(h)$.

Finally, we can raise up the time-frequency shift $\pi(x,\gamma) = E_\gamma T_x$ to
\[
\widetilde{\pi} : \Lp{1}{G} \tens \Lp{1}{\widehat{G}} \to \Bop(\Lp{2}{G}), \qquad
\widetilde{\pi}_\xi(f)(y) = \int_{\widehat{G}} \int_G \xi(x,\gamma) \pi(x,\gamma)(f)(y) \, \mathrm{d}x \mathrm{d}\gamma.
\]
When $\xi = g \tens h$, this becomes $\widetilde{\pi}_\xi(f) = \mathcal{M}_{F^{-1}(h)}(g \conv f) = \widetilde{E}_h \widetilde{T}_g(f)$.

Thus, between groups and convolution algebras, we have the following identifications
\begin{align*}
\lambda_x = T_x &\longleftrightarrow \widetilde{T}_g = \widetilde{\lambda}(g) \quad (x \in G,\; g \in \Lp{1}{G}), \\
\mathcal{M}_\gamma = E_\gamma &\longleftrightarrow \widetilde{E}_h = \mathcal{M}_{F^{-1}(h)} \quad (\gamma \in \widehat{G},\; h \in \Lp{1}{\widehat{G}}), \\
\pi((x,\gamma)) = E_\gamma T_x &\longleftrightarrow \widetilde{E}_{h} \widetilde{T}_{g} = \widetilde{\pi}(g \tens h).
\end{align*}
Consequently, Definitions~\ref{def:translation}, \ref{def:modulation} and~\ref{def:time-freq-shift} coincide with the classical lifted operators when $\G$ is a LCA group.

Note that, the quantum Fourier transform $\cF$ in Definition~\ref{def:fourier} is exactly $\cF(g)=\widetilde{\lambda}(g)$ for any $g\in L^1(G)$, and so $\cF(g)=F^{-1}\mathcal{M}_{F(g)}F$, which shows that $\cF$ goes back to $F$ when $\G$ is a LCA group (cf. \cite[proposition~5.1]{kahng2010fourier}).

For $f, g \in \Lp{2}{G}$, one has $J_{\varphi}f(x) = \overline{f(x)}$, $J_{\widehat{\varphi}}f(x) = \overline{f(x^{-1})}$ and $V(g \tens f)(x,t) = g(xt)f(t)$. Hence
\[
(\id \tens F)(J_{\widehat{\varphi}} \tens J_{\varphi}) V(g \tens J_{\varphi} f)(x,\gamma)
= \int_G f(t) \overline{g(x^{-1}t)} \overline{\gamma(t)} \, \mathrm{d}t
= \langle f \mid E_{\gamma} T_x g \rangle,
\]
which shows that Definition~\ref{def:STFT} reproduces the classical windowed Fourier transform on an LCA group.

In the same spirit, the quantum short-time Fourier transform can be viewed as the lift  of the classical STFT to the convolution algebra setting. Indeed, for any $h_1 \in \Lp{1}{G} \cap \Lp{2}{G}$ and $h_2 \in \Lp{1}{\widehat{G}} \cap \Lp{2}{\widehat{G}}$, the operators $\widetilde{T}_{h_1}$ and $\widetilde{E}_{h_2}$ are exactly $T_{h_1}^{\G}$ and $E_{h_2}^{\G}$, defined in Definitions~\ref{def:translation} and \ref{def:modulation}, respectively. 
Then we have that the lifted STFT is as follow:
\begin{align*}
\tilde{\phi}_g f(h_1, h_2)
&:=\int_{G} \int_{\widehat{G}} \langle f \mid h_1(x) h_2(\gamma) E_{\gamma} T_x g \rangle \, \mathrm{d}x \mathrm{d}\gamma \\
&= \int_{G} \int_{\widehat{G}} \langle f \mid E_{\gamma} T_x g \rangle \overline{h_1(x) h_2(\gamma)} \, \mathrm{d}x \mathrm{d}\gamma \\
&= \int_{G} f(t) \Bigl( \int_{\widehat{G}} \Bigl( \int_{G} \overline{g(x^{-1}t) h_1(x)} \, \mathrm{d}x \Bigr) \overline{\gamma(t) h_2(\gamma)} \, \mathrm{d}\gamma \Bigr) \mathrm{d}t \\
&= \int_G f(t) \, \overline{F^{-1}(h_2)(t) (h_1 \conv g)(t)} \, \mathrm{d}t \\
&= \langle f \mid \widetilde{E}_{h_2}\widetilde{T}_{h_1} g \rangle  
= \langle f \mid E_{h_2}^{\G} T_{h_1}^{\G} g \rangle 
= \langle \phi^{\G}_g f \mid h_1 \tens h_2 \rangle .
\end{align*}
Thus Theorem~\ref{prop:STFT-basic} provides precisely the $L^1$-lifted version of the short-time Fourier transform. Just as in the classical case, the quantum (lifted) STFT, regarded as an anti-linear functional on $\Lp{2}{\G} \tens \Lp{2}{\dualG}$, is completely determined by the quantum (lifted) time-frequency shift operator.
\end{example}

\begin{example}[Unimodular locally compact groups of type~I]\label{ex:ulcg_typeI}
Consider that $\G$ is a second-countable unimodular locally compact group $G$ of type~I with a fixed Haar measure $\mathrm{d}x$.
Denote by $\Irr(G)$ the set of equivalence classes of irreducible unitary representations of $G$, endowed with the Fell topology and the Plancherel measure $\mathrm{d}\pi$.  Here we still use $\pi$ to denote its equivalence class $[\pi]$. Recall that $HS(H_{\pi})$ is the Hilbert space of Hilbert-Schmidt operators on $H_\pi$, equipped with the inner product $\langle T\,|\,S\rangle_{HS(H_\pi)}=\Tr(S^*T)$, where $\Tr$ denotes the trace of an operator. 
The family $\{HS(H_\pi)\}_{[\pi]\in\Irr(G)}$ of Hilbert spaces indexed by $\Irr(G)$ is a measurable field of Hilbert spaces over $\Irr(G)$.
The Plancherel  theorem (see  \cite[Theorem 7.44]{Fo} or \cite[Section 2]{FK}) shows that, the Fourier transform, given by
\[
F f(\pi) = \widehat{f}(\pi) := \int_{G} f(x) \pi(x)^{*} \, \mathrm{d}\mu(x), \quad f \in \mathcal{I}^1(G):=\Lp{1}{G} \cap \Lp{2}{G}, 
\] 
can be extended to a unitary map \[F : \Lp{2}{G} \to \mathcal{H}^2(\widehat{\G}):=\int_{\Irr(G)}^{\oplus} HS(H_{\pi}) \, \mathrm{d}\pi,\] 
which intertwines the left regular representation $\lambda$ of $G$ with $\int_{\Irr(G)}^{\oplus} \id\otimes\bar\pi\, \mathrm{d}\pi$. Here $\bar\pi$ is the conjugated representation of $\pi$ on the conjugated Hilbert space $\overline{H}_\pi$, and the tensor product $\rho\otimes\pi$ of two unitary representations $\rho$ and $\pi$ corresponds to the unitary representation of $G$ on $HS(\overline{H}_\pi,H_\rho)$ given by $T\mapsto \rho(x)T\bar\pi(x)^*$ for any $T\in HS(\overline{H}_\pi,H_\rho)$ and $x\in G$.
Also, we denote by $\mathcal{I}^2(G)$ the finite $\mathbb{C}$-linear combinations of continuous of elements of $\mathcal{I}^1(G)$, and each $h\in\mathcal{I}^2(G)$ satisfies the Fourier inversion formula $h(x)=F^{-1}(\widehat{h})(x)= \int_{\Irr(G)} \Tr\bigl(\widehat{h}(\pi)\pi(x)\bigr)\,\mathrm{d}\pi$. The readers can refer to \cite{Fo} or \cite{FK} for details.

As a locally compact quantum group $\G$ , its quantum algebra $\Lp{\infty}{G}$ with the comultiplication
$\Delta f(x,y) = f(xy)$ is a commutative unimodular Kac algebra, and its dual is the group von Neumann algebra
\[\Lp{\infty}{\widehat\G}=\mathcal{L}(G) = \overline{\{ \lambda_{f} : \lambda_{f} (g)(t) = \int_{G} f(x) g(x^{-1}t) \, \mathrm{d}\mu(x) \}}^{w*}.\]

As same as in Example~\ref{ex:lcag}, in the present case, the Fourier transform in the quantum-group sense (Definition~\ref{def:fourier}) also can go back to the classical Fourier transform above, since for $T=\{T_\pi\}_{[\pi]\in\Irr(G)} \in\mathcal{H}^2(\widehat{\G})$, by \cite[proposition~5.1]{kahng2010fourier} and \cite[Theorem~7.37]{Fo}, one has
\[
\begin{aligned}
(F\circ \cF f\circ F^{-1})(T)
&= (F\circ \tilde\lambda(f)\circ F^{-1})(T) \\
&= \int_{\Irr(G)}^{\oplus} (\id\otimes \bar{\tilde\pi})(f)(T_\pi)\, \mathrm{d}\pi  = \int_{\Irr(G)}^{\oplus} T_\pi\tilde\pi(f)^*\, \mathrm{d}\pi  \\
&= \int_{\Irr(G)}^{\oplus} T_\pi \left(\int_G f(x)\pi(x)^* \mathrm{d}x \right)\, \mathrm{d}\pi \\
&= \int_{\Irr(G)}^{\oplus} T_\pi (F f)(\pi) \mathrm{d}\pi :=\mathcal{M}_{F f}(T).
\end{aligned}
\]

Hence, by Definition~\ref{def:STFT}, under the classical Fourier transform, the STFT of $f \in \Lp{2}{G}$ with window $g \in \Lp{2}{G}$ can be written as
\[
\begin{aligned}
(\phi_{g} f)(x, \pi)
&= (\id \tens F)\bigl( (1 \tens f) (S \tens \id)(\Delta(g^{*})) \bigr)(x, \pi) \\
&= \int_{G} f(t) \, \overline{g(x^{-1}t)} \, \pi(t)^{*} \, \mathrm{d}t \\
&= \int_{G}f(t)\bigl((E_{\pi}T_{x}g)(t)\bigr)^{*}\,\mathrm{d}t,
\end{aligned}
\]
where $T_x$ is the ordinary translation operator, and $E_{\pi}:\Lp{2}{G}\to \Lp{2}{G,\Bop(H_{\pi})}$ is the modulation operator defined by $E_{\pi}f(x)=f(x)\pi(x)$, for any $x \in G$, $\pi \in \Irr(G)$.  
Our definition of STFT coincides with the definition of the continuous Gabor transform in \cite[Definition~3.1]{FK}.

As the same as in the abelian case, the quantum translation operator  $T^{\G}_g$ in Definition~\ref{def:translation} is also $\widetilde{T}_g=\tilde\lambda(g)$ for any $g\in L^1(G)$.   
For any $h\in\mathcal{H}^2(\widehat{\G})\cap F(\Lp{1}{G})$,
the quantum modulation operator $E^{\G}_{h}$ in Definition~\ref{def:modulation} takes the form
\[
\begin{aligned}
E^{\G}_{h}f(x)=\widetilde{E}_{h}f(x)=\mathcal{M}_{F^{-1}h}f(x)
&= f(x)\int_{\Irr(G)} \Tr\bigl(h(\pi)\pi(x)\bigr)\,\mathrm{d}\pi  \\
&= \int_{\Irr(G)} \Tr\bigl(h(\pi)(E_{\pi}f)(x)\bigr)\,\mathrm{d}\pi.
\end{aligned}
\]

Thus, Similar to Example~\ref{ex:lcag}, for any $f,\, g\in\Lp{2}{G}$, $h_1 \in \Lp{1}{G} \cap \Lp{2}{G}$ and $h_2\in\mathcal{H}^2(\widehat{\G})\cap F(\Lp{1}{G})$, we also have
\begin{align*}
\langle \phi^{\G}_g f \mid h_1 \tens h_2 \rangle
&= \langle f \mid E_{h_2}^{\G} T_{h_1}^{\G} g \rangle= \langle f \mid \widetilde{E}_{h_2}\widetilde{T}_{h_1} g \rangle \\
&=\int_{\Irr(G)} \int_{G} \left\langle f \mid \Tr\bigl(h_1(x)(E_{\pi}T_{x}g)h_2(\pi)\bigr)\right\rangle_{\Lp{2}{G}} \, \mathrm{d}x \mathrm{d}\pi:=\tilde\phi_g f(h_1, h_2),
\end{align*}
which can be viewed as the $L^1$-lift of $\phi_g f$ onto the convolution algebra.
\end{example}

\bigskip
\section{Time-frequency shift as (co)projective representation}\label{sec4}
\medskip

This section explores the structural aspects of the time-frequency shift operator. Projective representations play a fundamental role in time-frequency analysis. In the classical setting, the short-time Fourier transform on a locally compact abelian group is built upon the Weyl-Heisenberg representation $\pi$. We demonstrate that, for a unimodular Kac algebra $\G$, the operator $\pi_{\G}$ introduced in Definition~\ref{def:time-freq-shift} also possesses a projective (co)representation structure associated with a specific 2-cocycle.

Let $\Omega$ be a 2-cocycle on $\G$. For $a, b \in \mathfrak{M}_{\varphi}$, there exists a unique element $a \conv_{\Omega} b := j^{-1}(\varphi_a \conv_{\Omega} \varphi_b)$. By Proposition~\ref{prop:twisted-convolution-regular} we have $\Lambda_{\varphi}(a \conv_{\Omega} b) = \lambda_{\G}^{\Omega}(\varphi_a) \Lambda_{\varphi}(b)$. Hence, for $a \in \Lp{1}{\G}$ and $b \in \Lp{2}{\G}$, we may define $a \conv_{\Omega} b = \lambda_{\G}^{\Omega}(\varphi_a)(b)$.

Write $\G \times \dualG$ for the direct product of $\G$ and its dual. Then $\Lp{\infty}{\G \times \dualG} = \Lp{\infty}{\G} \tens \Lp{\infty}{\dualG}$ is a locally compact quantum group with coproduct 
\[
\Delta_{\G \times \dualG} = (\id \tens \chi \tens \id) (\Delta_{\G} \tens \Delta_{\dualG})
\]
and invariant tracial weight $\varphi \tens \widehat{\varphi}$.

\begin{lemma}\label{lem:cocycle}
The operator $W_{14}^{\G*}$ is a unitary $2$-cocycle for $\G \times \dualG$.
\end{lemma}
\begin{proof}
We verify the cocycle condition~\eqref{eq:cocycle} directly. On the right-hand side, we have
\begin{align*}
&(1_{\G \times \dualG} \tens W_{14}^{\G*})(\id\tens \id \tens \Delta_{\G \times \dualG})(W_{14}^{\G*}) \\
&= W_{36}^{\G*} \bigl( \id\tens \id  \tens (\id \tens \chi \tens \id)(\Delta_{\G} \tens \Delta_{\dualG}) \bigr)(W_{14}^{\G*}) \\
&= W_{36}^{\G*} \bigl( (\id \tens \Delta_{\dualG})(W^{\G*}) \bigr)_{146}
   = W_{36}^{\G*} (W_{12}^{\G*} W_{13}^{\G*})_{146} \\
&= W_{36}^{\G*} W_{14}^{\G*} W_{16}^{\G*} .
\end{align*}
On the left-hand side, we have
\begin{align*}
&(W_{14}^{\G*} \tens 1_{\G \times \dualG})(\Delta_{\G \times \dualG} \tens \id\tens \id)(W_{14}^{\G*}) \\
&= W_{14}^{\G*} \bigl( (\Delta_{\G} \tens \id)(W^{\G*}) \bigr)_{136}
   = W_{14}^{\G*} (W_{23}^{\G*} W_{13}^{\G*})_{136} \\
&= W_{14}^{\G*} W_{36}^{\G*} W_{16}^{\G*}= W_{36}^{\G*} W_{14}^{\G*} W_{16}^{\G*} .
\end{align*}
In these computations above, we used the standard identities $(\Delta_{\G} \tens \id)(W^{\G}) = W_{13}^{\G} W_{23}^{\G}$ and $(\id \tens \Delta_{\dualG})(W^{\G}) = W_{13}^{\G} W_{12}^{\G}$. Hence it is clear that $W_{14}^{\G*}$ satisfies the cocycle condition.
\end{proof}

When $\G$ is a locally compact abelian group $G$,
the operator $W_{14}^{\G*}$ coincides with the classical $2$-cocycle $\sigma(\nu_1, \nu_2) = \overline{\gamma_2(x_1)}$ on $G \times \widehat{G}$, where $\nu_i = (x_i, \gamma_i) \in G \times \widehat{G}\;(i=1,2)$. 
It is well known that the Weyl-Heisenberg representation $\pi$ in \eqref{def:time-frequency shift} is a square-integrable, irreducible, $\sigma$-projective representation of $G\times\widehat{G}$, and so is its reflected representation $\pi_{re}(\nu) = \sigma(\nu,\nu) \pi(\nu) = T_x E_\gamma$, where $\sigma_r(\nu_1, \nu_2) = \overline{\sigma(\nu_2, \nu_1)}$ is the reflected cocycle. 
The following results show that $\pi_{\G}$ also has the same properties (see Corollary~\ref{cor:projective-rep}).

\begin{theorem}\label{thm:projective-corep}
The unitary $W_{23}^{\dualG} W_{13}^{\G} \in \Lp{\infty}{\G} \tens \Lp{\infty}{\dualG} \tens \Bop(\Lp{2}{\G})$ is a $W_{14}^{\G}$-projective, irreducible, square-integrable left corepresentation of $(\G \times \dualG)_{W_{14}^{\G*}}$ on the Hilbert space $\Lp{2}{\G}$.
\end{theorem}

\begin{proof}
From Lemma~\ref{lem:cocycle} we know that $W_{14}^{\G}$ is a 2-cocycle for the twisted quantum group $(\G \times \dualG)_{W_{14}^{\G*}}$. We first check condition~\eqref{eq:projective-corep} for the unitary $W_{23}^{\dualG} W_{13}^{\G}$ with respect to this cocycle:
\begin{align*}
&(\Delta_{\G \times \dualG} \tens \id)(W_{23}^{\dualG} W_{13}^{\G}) \cdot W_{14}^{\G} \\
&= (\id \tens \chi \tens \id \tens \id)
   \bigl( (\Delta_{\dualG} \tens \id)(W^{\dualG})_{345} (\Delta_{\G} \tens \id)(W^{\G})_{125} \bigr) \cdot W_{14}^{\G} \\
&= W_{25}^{\dualG} W_{45}^{\dualG} W_{15}^{\G} W_{35}^{\G} W_{14}^{\G} 
   = W_{25}^{\dualG} W_{45}^{\dualG} W_{15}^{\G} W_{14}^{\G} W_{35}^{\G} 
   = W_{25}^{\dualG} W_{15}^{\G} W_{45}^{\dualG} W_{35}^{\G} \\
&= (W_{23}^{\dualG} W_{13}^{\G})_{125} (W_{23}^{\dualG} W_{13}^{\G})_{345}.
\end{align*}
Here we again used $(\Delta_{\G} \tens \id)(W^{\G}) = W_{13}^{\G} W_{23}^{\G}$  together with the pentagon relation $W_{23}^{\G*} W_{12}^{\G} W_{13}^{\G} = W_{12}^{\G} W_{23}^{\G*}$. Therefore $W_{23}^{\dualG} W_{13}^{\G}$ is a $W_{14}^{\G}$-projective left corepresentation of $(\G \times \dualG)_{W_{14}^{\G*}}$.

To prove irreducibility, suppose $x \in \Bop(\Lp{2}{\G})$ satisfies
\[
W_{23}^{\dualG} W_{13}^{\G} ( 1_\G \tens 1_{\dualG} \tens x ) = (1_\G \tens 1_{\dualG} \tens x) W_{23}^{\dualG} W_{13}^{\G} .
\]
Then for all $\omega \in M_*$ and $\theta \in \widehat{M}_*$ we obtain
\[
\lambda_{\dualG}(\theta) \lambda_{\G}(\omega) x = x \lambda_{\dualG}(\theta) \lambda_{\G}(\omega),
\]
which implies $x \in \Lp{\infty}{\G}' \cap \Lp{\infty}{\dualG}',$ 
and then it directly follows from  \cite[Corollary 4.1.5]{Kac}) that $x = c 1$ for some scalar $c$. Hence the corepresentation is irreducible.

Square-integrability will be established in Corollary~\ref{cor:square-integrable} below.
\end{proof}

In the classical Gabor analysis, the $L^1$-lift $\widehat{\pi}$ of the time-frequency shift representation $\pi$ in \eqref{def:time-frequency shift} is not a homomorphism under the canonical convolution in $L^1(G\times\widehat{G})$, but can preserve the twisted convolution. 
Moreover, under the canonical involution $*$, this representation and its reflected representation satisfy that $\widehat{\pi}(f^*) = \widehat{\pi}_{re}(f)^*$ for any $f\in L^1(G\times\widehat{G})$. 
The next proposition tells us that these basic phenomenons mentioned above also exist in quantum group theory.

\begin{proposition}\label{prop:TF-properties}
For every $\nu \in \Lp{1}{\G} \tens \Lp{1}{\dualG}$ we have
\begin{align*}
\pi^{re}_{\G}(\nu) &= \pi_{\G}(\nu^{\sharp })^* = \pi_{\G}(W^{\G} \nu), \\
\pi_{\G}(\nu) &= \pi_{\G}^{re}(\mu^{\sharp  })^* = \pi^{re}_{\G}(W^{\G*} \nu).
\end{align*}

Moreover, for all $\nu, \mu \in \Lp{1}{\G} \tens \Lp{1}{\dualG}$,
\[
\pi_{\G}(\nu \conv_{W_{14}^{\G*}} \mu) = \pi_{\G}(\nu) \pi_{\G}(\mu).
\]
\end{proposition}

\begin{proof}
Using the pentagon equation $W^{\G}_{12} W^{\G}_{13} W^{\G}_{23} = W^{\G}_{23} W^{\G}_{12}$ we compute
\begin{align*}
\pi^{re}_{\G}(\nu) &= (\varphi \tens \widehat{\varphi})_{\nu13}(W^{\G}_{12} W^{\G*}_{23}) 
                         = (\varphi \tens \widehat{\varphi})_{\nu13}(W^{\G*}_{23} W^{\G}_{12} W^{\G}_{13}) 
                         = \pi_{\G}(W^{\G} \nu), \\
\pi_{\G}(\nu)   &= (\varphi \tens \widehat{\varphi})_{\nu13}(W^{\G*}_{23} W^{\G}_{12}) 
                         = (\varphi \tens \widehat{\varphi})_{\nu13}(W^{\G}_{12} W^{\G*}_{23} W^{\G*}_{13}) 
                         = \pi^{re}_{\G}(W^{\G*} \nu).
\end{align*}
Since $\cFp{1}(x)^* = \cFp{1}(x^{\sharp })$ for $x \in \Lp{1}{\G}$, we also have $\pi_{\G}(\nu^{\sharp})^* = \pi_{\G}^{re}(\nu)$ and $\pi_{\G}^{re}(\nu^{\sharp  })^* = \pi_{\G}(\nu)$. The multiplicativity property follows directly from (\ref{prop:proj-corep-homomorphism}) together with Theorem~\ref{thm:projective-corep}.
\end{proof}

\begin{corollary}\label{cor:projective-rep}
The map $\pi_{\G} : \Lp{1}{\G \tens \dualG} \to \Bop(\Lp{2}{\G})$, $\nu \mapsto ((\varphi \tens \widehat{\varphi})_\nu \tens \id)(W_{23}^{\dualG} W_{13}^{\G})$, extends to a continuous $W_{14}^{\G}$-projective left representation of the dual quantum group of $(\G \times \dualG)_{W_{14}^{\G*}}$.
\end{corollary}

\begin{remark}\label{rem:quantum-double}
According to the argument before Definition~9.2.1 in \cite{De phd}, the operator $W^{\dualG*}_{23} = (\Sigma W^{\G} \Sigma)_{23}$ is also a 2-cocycle for $\G \times \dualG$. In complete analogy with Theorem~\ref{thm:projective-corep}, one can show that $W_{13}^{\G} W_{23}^{\dualG}$ is a $W^{\dualG}_{23}$-projective, irreducible, square-integrable left corepresentation of $(\G \times \dualG)_{W_{23}^{\dualG*}}$ on $\Lp{2}{\G}$. Consequently, the map $\pi^{re}_{\G}$ extends to a continuous  $W_{23}^{\dualG}$-projective left representation of the dual quantum group of $(\G \times \dualG)_{W_{23}^{\dualG*}}$. The latter twisted quantum group is precisely the \emph{quantum double} of $\G$ and $\dualG$ with respect to $W^{\G}$ (cf.  \cite[Definition 9.2.1]{De phd}).
\end{remark}

\bigskip
\section{Covariance principle}\label{sec5}
\medskip

This section establishes two central structural properties of the short-time Fourier transform on unimodular Kac algebras: a fundamental identity relating the STFT on $\G$ to that on its dual $\dualG$, and a covariance principle that describes how the STFT intertwines the time-frequency shift operator with a projective regular representation. Both of this two results are direct generalizations of their classical counterparts in abelian time-frequency analysis.

\begin{theorem}[Fundamental identity]\label{thm:fundamental-identity}
For any $x, y \in \Lp{2}{\G}$,
\[
\phi_y^{\G} x = W^{\G} \Sigma \; \phi_{\cF y}^{\dualG} \cF x.
\]
\end{theorem}

\begin{proof}
Take $x, y \in \Lp{2}{\G}$, $a \in \mathfrak{M}_{\varphi}$ and $b \in \mathfrak{M}_{\widehat{\varphi}}$.
Using Theorem~\ref{prop:STFT-basic}, Remark~\ref{rem:unitary-equivalence} and Proposition~\ref{prop:TF-properties}, we compute
\begin{align*}
\bigl\langle \phi_y^{\G} x \bigm| \Lambda_{\varphi}(a) \tens \Lambda_{\widehat{\varphi}}(b) \bigr\rangle 
&= \bigl\langle x \bigm| E^{\G}_b T^{\G}_a y \bigr\rangle \\
&= \bigl\langle x \bigm| \pi^{re}_{\G}(W^{\G*} a \tens b)(y) \bigr\rangle \\
&= \bigl\langle \cF x \bigm| \pi_{\dualG}\bigl( \chi(W^{\G*} a \tens b) \bigr)(\cF y) \bigr\rangle \\
&= \bigl\langle \phi_{\cF y}^{\dualG} \cF x \bigm| \Sigma W^{\G*} \Lambda_{\varphi}(a) \tens \Lambda_{\widehat{\varphi}}(b) \bigr\rangle \\
&= \bigl\langle W^{\G} \Sigma \phi_{\cF y}^{\dualG} \cF x \bigm| \Lambda_{\varphi}(a) \tens \Lambda_{\widehat{\varphi}}(b) \bigr\rangle .
\end{align*}
Since $\mathfrak{M}_{\varphi}$ is dense in $\Lp{2}{\G}$, the fundamental identity follows.
\end{proof}

\begin{theorem}[Covariance principle]\label{thm:covariance}
For any $a \in \Lp{1}{\G} \tens \Lp{1}{\dualG}$ and $x, y \in \Lp{2}{\G}$, one has
\[
\phi^{\G}_y \bigl( \pi_{\G}(a) x \bigr)
= \lambda_{\G \times {\dualG}}^{W_{14}^{\G*}} \bigl( (\varphi \tens \widehat{\varphi})_a \bigr) \; \phi^{\G}_y x.
\]
\end{theorem}

\begin{proof}
First we prove that for every $\nu \in \Lp{\infty}{\G}_* \tens \Lp{\infty}{\dualG}_*$, one has
\begin{equation}\label{prop:involution-twisted-regular}
    \lambda_{\G \times {\dualG}}^{W_{14}^{\G*}}(\nu)^*
= \lambda_{\G \times {\dualG}}^{W_{14}^{\G*}}(W^{\G} \nu^\sharp).
\end{equation}
Because for $\mu \in \Lp{\infty}{\G}_* \tens \Lp{\infty}{\dualG}_*$, we have
\begin{align*}
\mu\bigl( \lambda_{\G \times {\dualG}}^{W_{14}^{\G*}}(W^{\G} \nu^{\sharp}) \bigr)
&= \mu\bigl(( \nu^\sharp \tens \id \tens \id )(W^{\G \times \dualG} W_{14}^{\G} W_{12}^{\G})\bigr) \\
&= \bar{\nu} \bigl(S_{\G\times \dualG} ((\id\tens \id \tens  \mu)(W^{\G\times \dualG}W_{14}^{\G} W_{12}^{\G})) \bigr),
\end{align*}
while 
\[
\mu\bigl( \lambda_{\G \times {\dualG}}^{W_{14}^{\G*}}(\nu)^* \bigr)
= \mu\bigl( 
((\nu \tens \id \tens \id )(W^{\G\times \dualG} W_{14}^{\G}))^* \bigr)
= \bar{\nu}\bigl(\id \tens  \id \tens \mu)(W_{14}^{\G*} W^{\G\times \dualG*}) \bigr).
\]

Thus it suffices to prove
\begin{equation}\label{eq:S-G-dual}
S_{\G\times \dualG} \bigl((\id \tens \id \tens  \mu)(W^{\G\times \dualG}W_{14}^{\G} W_{12}^{\G})\bigr)
=(\id \tens \id \tens \mu)(W_{14}^{\G*} W^{\G\times \dualG*}).
\end{equation}

Indeed, choose $c,d,e,f\in \mathfrak{N}_{\varphi\tens \widehat{\varphi}}$. Then we have
\begin{align*}
&\bigl\langle \bigl((\id \tens \id \tens \omega_{\Lambda_{\varphi\tens \widehat{\varphi}}(c),\Lambda_{\varphi\tens \widehat{\varphi}}(d)} ) 
( W^{\G\times \dualG} W_{14}^{\G} W_{12}^{\G} ) \bigr)
\Lambda_{\varphi\tens \widehat{\varphi}}(e)  \bigm| \Lambda_{\varphi\tens \widehat{\varphi}}(f) \bigr\rangle \\
&= \bigl\langle W^{\G\times \dualG} W_{14}^{\G} W_{12}^{\G}
(\Lambda_{\varphi\tens \widehat{\varphi}}(e)\tens \Lambda_{\varphi\tens \widehat{\varphi}}(c)) \bigm| 
\Lambda_{\varphi\tens \widehat{\varphi}}(f)\tens \Lambda_{\varphi\tens \widehat{\varphi}}(d)\bigr\rangle \\
&= \bigl\langle W_{14}^{\G} W_{12}^{\G} (\Lambda_{\varphi\tens \widehat{\varphi}}(e)\tens \Lambda_{\varphi\tens \widehat{\varphi}}(c) ) \bigm| 
(\Lambda_{\varphi\tens \widehat{\varphi}}\tens \Lambda_{\varphi\tens \widehat{\varphi}})(\Delta_{\G\times \dualG}(d)(f\tens 1_{\G\times \dualG})) \bigr\rangle \\
&= \bigl\langle (\id \tens \id \tens  \varphi \tens \widehat{\varphi})
\bigl( \Delta_{\G\times \dualG}(d^*)  W_{14}^{\G} W_{12}^{\G} (1_{\G\times \dualG} \tens c ) \bigr)
\Lambda_{\varphi\tens \widehat{\varphi}}(e) \bigm| \Lambda_{\varphi\tens \widehat{\varphi}}(f) \bigr\rangle ,
\end{align*}
which implies 
\[
\bigl((\id \tens \id \tens\omega_{\Lambda_{\varphi\tens \widehat{\varphi}}(c),\Lambda_{\varphi\tens \widehat{\varphi}}(d)} ) 
( W^{\G\times \dualG} W_{14}^{\G} W_{12}^{\G} ) \bigr) 
= (\id \tens \id \tens  \varphi \tens \widehat{\varphi})
\bigl( \Delta_{\G\times \dualG}(d^*)  W_{14}^{\G} W_{12}^{\G} (1_{\G \times \dualG} \tens c ) \bigr).
\]

Observe that
\[
(\Delta_{\G} \tens \Delta_{\dualG})(W^{\G}) 
= (\id \tens \id \tens \Delta_{\dualG})(W_{13}^{\G} W_{23}^{\G}) 
= W_{14}^{\G} W_{13}^{\G} W_{24}^{\G} W_{23}^{\G},
\]
so that $W_{14}^{\G} W_{12}^{\G} = (\Delta_{\G\times\dualG})(W^{\G})(\id \tens \Sigma \tens \id) (\id \tens \Delta_{\dualG})(W^{\G*})_{234}(\id \tens \Sigma \tens \id)$. Consequently, we have
\begin{align*}
&(\id \tens \id \tens \omega_{\Lambda_{\varphi\tens \widehat{\varphi}}(c),\Lambda_{\varphi\tens \widehat{\varphi}}(d)} ) 
( W^{\G\times \dualG} W_{14}^{\G} W_{12}^{\G} ) \\
&= (\id \tens \id \tens \varphi \tens \widehat{\varphi})
\bigl( \Delta_{\G\times\dualG}(d^*W^{\G})
(\id \tens \Sigma \tens \id)(\id \tens \Delta_{\dualG})(W^{\G*})_{234 } (\id \tens \Sigma \tens \id)(1_{\G \times \dualG} \tens c) \bigr).
\end{align*}

Because $S_{\G\times \dualG}=S_{\G} \tens S_{\dualG}$, using \cite[Proposition~1.6.17]{vaes2001locally}, we obtain

\begin{align*}
&S_{\G\times\dualG}\Bigl(
(\id \tens \id \tens \varphi \tens \widehat{\varphi})
\bigl( \Delta_{\G\times\dualG}(d^*W^{\G})
(\id \tens \Sigma \tens \id)(\id \tens \Delta_{\dualG})(W^{\G*})_{234 } (\id \tens \Sigma \tens \id)(1_{\G \times \dualG} \tens c) \bigr)\Bigr) \\
&= (\id \tens \id \tens \varphi  \tens \widehat{\varphi})
\bigl( (1_{\G \times \dualG} \tens d^* ) W_{34}^{\G} (\Delta_{\G} \tens \id)(W^{\G*})_{134} \Delta_{\G\times \dualG}(c)  \bigr) \\
&= (\id \tens \id \tens \varphi  \tens \widehat{\varphi})
\bigl( (1_{\G \times \dualG} \tens d^* ) W_{14}^{\G*} \Delta_{\G\times \dualG}(c) \bigr).
\end{align*}

A similar computation yields
\[
(\id \tens \id \tens \omega_{\Lambda_{\varphi\tens \widehat{\varphi}}(c),\Lambda_{\varphi\tens \widehat{\varphi}}(d)} )
\bigl( W_{14}^{\G*} W^{\G\times\dualG*} \bigr) = (\id \tens \id \tens \varphi\tens \widehat{\varphi})
\bigl( (1_{\G \times \dualG} \tens d^* ) W_{14}^{\G*} \Delta_{\G\times \dualG}(c)  \bigr).
\]

Therefore, we have
\[
S_{\G\times \dualG}\bigl((\id \tens \id \tens  \omega_{\Lambda_{\varphi\tens \widehat{\varphi}}(c),\Lambda_{\varphi\tens \widehat{\varphi}}(d)} ) 
( W^{\G\times \dualG} W_{14}^{\G} W_{12}^{\G} ) \bigr) 
= (\id \tens \id \tens \omega_{\Lambda_{\varphi\tens \widehat{\varphi}}(c),\Lambda_{\varphi\tens \widehat{\varphi}}(d)} ) 
(W_{14}^{\G*} W^{\G\times \dualG*}) .
\]

Since the linear spans
$\{ \omega_{\Lambda_{\varphi\tens \widehat{\varphi}}(c),\Lambda_{\varphi\tens \widehat{\varphi}}(d)} \mid c,d \in \mathfrak{N}_{\varphi\tens \widehat{\varphi}}\}$
is norm-dense in $\Bop(\Lp{2}{\G\times \dualG})_*$, the equation \ref{eq:S-G-dual} holds, and so the desired identity (\ref{prop:involution-twisted-regular}) follows.

Next, choose $b \in \Lp{1}{\G \times \dualG} \cap \Lp{2}{\G \times \dualG}$ and $x, y \in \Lp{2}{\G}$.
Using Proposition~\ref{prop:twisted-convolution-regular}, Proposition~\ref{prop:TF-properties} and (\ref{prop:involution-twisted-regular}), we compute
\begin{align*}
\bigl\langle b \bigm| \phi^{\G}_y \pi_{\G}(a) x \bigr\rangle 
&= \bigl\langle \pi_{\G}(b)(y) \bigm| \pi_{\G}(a) x \bigr\rangle \\
&= \bigl\langle \pi^{re}_{\G}(a^{\sharp}) \pi_{\G}(b)(y) \bigm| x \bigr\rangle \\
&= \bigl\langle \pi_{\G} \bigl( (W^{\G} a^{\sharp}) \conv_{W_{14}^{\G*}} b \bigr)(y) \bigm| x \bigr\rangle \\
&= \bigl\langle \lambda_{\G \times {\dualG}}^{W_{14}^{\G*}} \bigl( (\varphi \tens \widehat{\varphi})_{W^{\G} a^{\sharp}} \bigr)(b) \bigm| \phi^{\G}_y x \bigr\rangle \\
&= \bigl\langle b \bigm| \lambda_{\G \times {\dualG}}^{W_{14}^{\G*}} \bigl( (\varphi \tens \widehat{\varphi})_a \bigr) \phi^{\G}_y x \bigr\rangle .
\end{align*}
Because $\Lp{1}{\G \times \dualG} \cap \Lp{2}{\G \times \dualG}$ is dense in $\Lp{2}{\G} \tens \Lp{2}{\dualG}$, the equality holds as claimed.
\end{proof}

\begin{corollary}[Square-integrability]\label{cor:square-integrable}
Let $y$ be a unit vector in $\Lp{2}{\G}$. Then the short-time Fourier transform $\phi_y^{\G}$ is an isometric intertwiner between the $W_{14}^{\G}$-projective left corepresentations $W_{23}^{\dualG} W_{13}^{\G}$ and $W^{\G \times \dualG} W_{14}^{\G}$ on $\Lp{2}{\G}$ and $\Lp{2}{\G} \tens \Lp{2}{\dualG}$, respectively. Consequently, $W_{23}^{\dualG} W_{13}^{\G}$ is a sub-corepresentation of $W^{\G \times \dualG} W_{14}^{\G}$ and is therefore square-integrable.
\end{corollary}

\begin{proof}
Take $x, y \in \Lp{2}{\G}$ with $\| y \|_2 = 1$, and let $a,b,c \in \Lp{2}{\G\times \dualG }$.
By Theorem~\ref{thm:covariance}, we have
\begin{align*}
&\bigl\langle (1_{\G\times \dualG}  \tens \phi_y^{\G}) \bigl( (W_{23}^{\dualG} W_{13}^{\G})(a\tens x)\bigr) \bigm| b\tens c \bigr \rangle \\
&=\bigl\langle \phi_y^{\G} \bigl( (\omega_{a,b} \tens \id)(W_{23}^{\dualG} W_{13}^{\G}) x \bigr) \bigm| c \bigr\rangle 
= \bigl\langle\phi_y^{\G} \bigl( \pi_{\G}(j^{-1}(\omega_{a,b})) x \bigr)  \bigm| c \bigr\rangle \\
&=\bigl\langle \lambda_{\G \times {\dualG}}^{W_{14}^{\G*}} ( \omega_{a,b} ) \; \phi^{\G}_y x  \bigm| c \bigr\rangle 
=\bigl\langle \bigl( (\omega_{a,b} \tens \id)(W^{\G \times \dualG} W_{14}^{\G}) \bigr) (\phi_y^{\G} x)  \bigm| c \bigr\rangle \\
&= \bigl\langle \bigl(W^{\G \times \dualG} W_{14}^{\G}\bigr) ( (1_{\G\times \dualG}  \tens \phi_y^{\G}) (a\tens x) ) \bigm| b\tens c \bigr \rangle.
\end{align*}
Consequently, it follows that
\[
(1_{\G\times \dualG}  \tens \phi_y^{\G}) (W_{23}^{\dualG} W_{13}^{\G})
= W^{\G \times \dualG} W_{14}^{\G} (1_{\G\times \dualG} \tens \phi_y^{\G}).
\]
By Theorem~\ref{prop:STFT-basic}, we know that $\phi_y^{\G}$ is an isometry, and hence $W_{23}^{\dualG} W_{13}^{\G}$ is a sub-corepresentation of $W^{\G \times \dualG} W_{14}^{\G}$.
Since the projective left regular representation $W^{\G \times \dualG} W_{14}^{\G}$ is square-integrable (see Remark~\ref{rem:sqr-int-galoismap}), the corepresentation $W_{23}^{\dualG} W_{13}^{\G}$ as sub-corepresentation of $W^{\G \times \dualG} W_{14}^{\G}$ inherits this property.
\end{proof}

% Returning to the classical setting, the covariance principle can also be understood at the level of convolution algebras. Let $G$ be an LCA group. Define the $\sigma$-twisted left regular representation on $\Lp{2}{G \times \widehat{G}}$ by
% \[
% \widehat{\lambda}_{G \times \widehat{G}}^{\sigma} : \Lp{1}{G \times \widehat{G}} \to \Bop(\Lp{2}{G \times \widehat{G}}), \quad
% \widehat{\lambda}_{G \times \widehat{G}}^{\sigma}(h) F (\nu_2) = (h \conv_\sigma F)(\nu_2).
% \]
%
% For $f, g \in \Lp{2}{G}$ and $h \in \Lp{1}{G \times \widehat{G}}$, one has on one hand
% \[
% \int_{G \times \widehat{G}} \phi_g(\pi(\nu_1) f)(\nu_2) h(\nu_1) \, \mathrm{d}\nu_1
% = \langle \widehat{\pi}(h) f \mid \pi(\nu_2) g \rangle
% = \phi_g(\widehat{\pi}(h) f)(\nu_2),
% \]
% and on the other hand
% \[
% \int_{G \times \widehat{G}} \lambda_{G \times \widehat{G}}^{\sigma}(\nu_1) \phi_g f(\nu_2) h(\nu_1) \, \mathrm{d}\nu_1
% = \widehat{\lambda}_{G \times \widehat{G}}^{\sigma}(h)(\phi_g f)(\nu_2).
% \]
% Since $\phi_g(\pi(\nu_1) f)(\nu_2) = \lambda_{\sigma}(\nu_1) \phi_g f(\nu_1^{-1} \nu_2)$, we obtain
% \[
% \phi_g(\widehat{\pi}(h) f) = \widehat{\lambda}_{G \times \widehat{G}}^{\sigma}(h) (\phi_g f),
% \]
% which is precisely the classical counterpart of Theorem~\ref{thm:covariance} expressed in the language of convolution algebras.

\bigskip
\section{Uncertainty principles and Kac-Paljutkin quantum group}\label{sec6}
\medskip

This section extends several classical uncertainty principles to the setting of unimodular Kac algebras. Uncertainty principles lie at the heart of time-frequency analysis, quantifying the inherent trade-off between localisation in time and frequency. We first prove a Lieb-type inequality for the short-time Fourier transform, then establish entropy-based uncertainty relations,  and finally illustrate the previous results with Kac--Paljutkin finite quantum group.

Let $\G$ be a unimodular Kac algebra with trace $\varphi$. For $x \in \Lp{p}{\G},\; 1 \le p \le \infty$, denote by $\mathcal{S}(x) = \varphi(\range(x))$ the trace of the range projection of $x$. The Donoho--Stark uncertainty principle for unimodular Kac algebras \cite[Proposition~3.3]{Liu2017} states that for any non-zero $x \in \Lp{1}{\G} \cap \Lp{2}{\G}$,
\[
\mathcal{S}(x) \, \mathcal{S}(\cF x) \ge 1 .
\]
We shall show that, this kind of result also exists for the short-time Fourier transform (see Theorem~\ref{thm:uncertainty-STFT} below).

Firstly, in the classical setting, a convenient substitute for an uncertainty principle valid on every LCA group is provided by Lieb's inequalities for the STFT.  Recall that any LCA group $G$ is topologically isomorphic to $\mathbb{R}^{d} \times G_{0}$, where $d \ge 0$ and $G_{0}$ contains a compact open subgroup \cite[Theorem~24.30]{LCA1970}. Lieb's inequality \cite[Theorem 6.3.1]{Gro1998} states that for $f, g \in \Lp{2}{G}$, there exists
\[
\| \phi_{g} f \|_{\Lp{p}{G \times \widehat{G}}} \le \Bigl( \frac{2}{p} \Bigr)^{d/p} \| f \|_{\Lp{2}{G}} \| g \|_{\Lp{2}{G}}, \qquad 2 \le p < \infty .
\]
Hence, before the uncertainty principle for Kac algebras is given, we also need to establish the inequality of this type on unimodular Kac algebras.

\begin{theorem}[Lieb's uncertainty inequality]\label{thm:lieb-up}
Let $\G$ be a unimodular Kac algebra. For every $x, y \in \mathfrak{N}_{\varphi}$, we have
\[
\| \phi^{\G}_{y} x \|_{p} \le \|x\|_2 \, \| y \|_2 \qquad 2 \le p < \infty,
\]
and
\[
\| \phi^{\G}_{y} x \|_{p} \ge \|x\|_2 \, \| y \|_2\qquad 0 < p \le 2.
\]
\end{theorem}

\begin{proof}
For every $x, y \in \mathfrak{N}_{\varphi}$, by Remark~\ref{rmk:another def of STFT}, we know that
$\phi^{\G}_{y} x \in \mathfrak{N}_{\varphi \tens \widehat{\varphi}}$.

On one hand,  we estimate the $L^{\infty}$-norm:
\begin{align*}
\| \phi^{\G}_{y} x \|_{\infty}
&= \| (\phi^{\G}_{y} x)^{*} \|_{\infty}
 = \sup_{\substack{\nu \in \mathfrak{M}_{\varphi \otimes \widehat{\varphi}}\\ \| \nu \|_{1} \le 1}}
   \bigl| (\varphi \tens \widehat{\varphi}) \bigl( \nu (\phi^{\G}_{y} x)^{*} \bigr) \bigr| \\
&= \sup_{\substack{\nu \in \mathfrak{M}_{\varphi \otimes \widehat{\varphi}}\\ \| \nu \|_{1} \le 1}}
   \bigl| \langle (\Lambda_{\varphi}\otimes \Lambda_{\widehat{\varphi}})(\nu) \mid \phi^{\G}_{\Lambda_{\varphi}(y)} \Lambda_{\varphi}(x) \rangle \bigr| \\
&= \sup_{\substack{\nu \in \mathfrak{M}_{\varphi \otimes \widehat{\varphi}}\\ \| \nu \|_{1} \le 1}}
   \bigl| \langle \pi_{\G}(\nu)(\Lambda_{\varphi}(y)) \mid \Lambda_{\varphi}(x) \rangle \bigr| \\
&\le \sup_{\substack{\mathfrak{M}_{\varphi \otimes \widehat{\varphi}}\\ \| \nu \|_{1} \le 1}}
   \| \nu \|_{1} \, \| x \|_2 \, \| y \|_2
   \le \| x \|_2 \, \| y \|_2.
\end{align*}
Together with the Plancherel identity $\| \phi^{\G}_{y} x \|_{2} = \| x \|_2 \, \| y \|_2$ (by Theorem~\ref{prop:STFT-basic}), we now obtain
\[
\| \phi^{\G}_{y} x \|_{2} = \| x \|_2 \, \| y \|_2, \qquad 
\| \phi^{\G}_{y} x \|_{\infty} \le \| x \|_2 \, \| y \|_2 .
\]
%Then, applying the standard complex interpolation method to the operator $\phi^{\G}_{y}$ (see e.g. \cite[Section 2.1]{caspers2013p}),  it is easily proved that $\| \phi^{\G}_{y} x \|_{p} \le \| x \|_2 \, \| y \|_2$ for all $p\in [2, \infty)$.

%On the other hand, for the case $p<2$,
We consider the self-adjoint element $\frac{|\phi_y^{\G}x|}{\|x\|_2\|y\|_2}$ whose $L^\infty$-norm is not more that $1$ by the above estimation. 
By functional calculus, we obtain that 
\[\left(\frac{|\phi_y^{\G}x|}{\|x\|_2\|y\|_2}\right)^p \le \left(\frac{|\phi_y^{\G}x|}{\|x\|_2\|y\|_2}\right)^2, \quad 2\le p < \infty ,\]
and
\[\left(\frac{|\phi_y^{\G}x|}{\|x\|_2\|y\|_2}\right)^p \ge \left(\frac{|\phi_y^{\G}x|}{\|x\|_2\|y\|_2}\right)^2, \quad 0<p \le 2,\]
which imply that $\| \phi^{\G}_{y} x \|_{p} \le  \|x\|_2 \, \| y \|_2$ for all $p\in  [2,\infty)$ and  $\| \phi^{\G}_{y} x \|_{p} \ge \|x\|_2 \, \| y \|_2$ for all $p\in(0, 2].$

\end{proof}

Next, in time-frequency analysis, spectral entropy measures how the energy of a signal is distributed over the time-frequency plane; a low entropy indicates a sharply concentrated signal, which corresponds to an ``optimal'' window.  In the quantum setting, the von Neumann entropy, widely used in quantum information theory, naturally replaces the classical Shannon entropy.

For a $\varphi$-measurable element $x$, the von Neumann entropy of $| x |^{2}$ is defined as
\[
\entropy(| x |^{2}) := -\varphi\bigl( x^{*} x \log(x^{*} x) \bigr).
\]

\begin{proposition}[Entropy inequality]\label{prop:entropy}
Let $\G$ be a unimodular Kac algebra. For any $x, y \in \Lp{2}{\G}$, it holds that
\[
\entropy\bigl( | \phi^{\G}_{y} x |^{2} \bigr) \ge
-2 \, \| x \|_{2}^{2} \, \| y \|_{2}^{2} \bigl( \log \| x \|_{2} + \log \| y \|_{2} \bigr).
\]
\end{proposition}

\begin{proof}
For $x, y \in \Lp{2}{\G}$ choose nets $\{ x_{\alpha} \} \subset \mathfrak{M}_{\varphi}$ and
$\{ y_{\alpha} \} \subset \mathfrak{M}_{\varphi}$ such that
$\lim_{\alpha} \| \Lambda_{\varphi}(x_{\alpha}) - x \| = \lim_{\alpha} \| \Lambda_{\varphi}(y_{\alpha}) - y \| = 0$.
Then $\lim_{\alpha} \| \phi^{\G}_{\Lambda_{\varphi}(y_{\alpha})} \Lambda_{\varphi}(x_{\alpha}) - \phi^{\G}_{y} x \| = 0$.
By \cite[Remark~5.14]{Liu}, the entropy is continuous with respect to this convergence, and then
$\entropy( | \phi^{\G}_{y} x |^{2} ) = \lim_{\alpha} \entropy( | \phi^{\G}_{\Lambda_{\varphi}(y_{\alpha})} \Lambda_{\varphi}(x_{\alpha}) |^{2} )$.
Thus, we may assume $x, y \in \mathfrak{M}_{\varphi}\subset\mathfrak{N}_{\varphi}$, and in this case, we have $\phi^{\G}_{y} x \in \mathfrak{N}_{\varphi \tens \widehat{\varphi}}$ by Remark~\ref{rmk:another def of STFT}.

Define $F(p) = \| \phi^{\G}_{y} x \|_{p} - \| x \|_{2} \| y \|_{2}$ for $p \ge 2$. By  \cite[Section II, Lemma~18]{terp}, we have that $p\mapsto | x |^{p}$ is differentiable for $p>0$ and then can differentiate $F$ at $p = 2$:
\begin{align*}
F'(2) &= -\frac{1}{2} \| \phi^{\G}_{y} x \|_{2} \log \| \phi^{\G}_{y} x \|_{2}
      -\frac{1}{4} \, \frac{\entropy( | \phi^{\G}_{y} x |^{2} )}{\| x \|_{2} \| y \|_{2}} \\
     &= -\frac{1}{2} \| x \|_{2} \| y \|_{2} \bigl( \log \| x \|_{2} + \log \| y \|_{2} \bigr)
      -\frac{1}{4} \, \frac{\entropy( | \phi^{\G}_{y} x |^{2} )}{\| x \|_{2} \| y \|_{2}} .
\end{align*}
Both of Theorem~\ref{prop:STFT-basic} and the first inequality in Theorem~\ref{thm:lieb-up}  give $F(p) \le 0$ and $F(2) = 0$, whence $F'(2) \le 0$ and directly implies the claimed lower bound for $\entropy( | \phi^{\G}_{y} x |^{2} )$.
\end{proof}

Now, it is time to present the Donoho--Stark uncertainty principle for unimodular Kac algebras.

\begin{theorem}[Donoho--Stark uncertainty principle for the STFT]\label{thm:uncertainty-STFT}
For any $x, y \in \Lp{2}{\G}$,
\begin{equation}\label{eq:uncertainty-STFT}
\mathcal{S}\bigl( \phi^{\G}_{y} x \bigr) \ge 1 .
\end{equation}
\end{theorem}

\begin{proof}
We may assume $\| x \|_{2} = \| y \|_{2} = 1$.  For a unimodular Kac algebra, one has the general inequality
$\log \mathcal{S}(z) \ge \entropy(| z |^{2})$ for any non-zero $z \in \Lp{2}{\G}$ (see the first proof of \cite[Proposition~3.3]{Liu2017}).
Then, Combined with Proposition~\ref{prop:entropy}, the inequality above  yields
\[
\log \mathcal{S}\bigl( \phi^{\G}_{y} x \bigr) \ge \entropy\bigl( | \phi^{\G}_{y} x |^{2} \bigr) \ge 0 ,
\]
which is equivalent to \eqref{eq:uncertainty-STFT}.
\end{proof}

At the end of this paper, we will give an example that provides a non‑trivial, non‑commutative and non‑cocommutative unimodular Kac algebra for which the equality case $\mathcal{S}(\phi^{\G}_{y}x)=1$ in \eqref{eq:uncertainty-STFT} can actually occur.

\begin{example}[Kac--Paljutkin quantum group]
Consider the eight-dimensional Kac--Paljutkin finite quantum group $\mathbb{KP}$, the smallest Hopf--von Neumann algebra that is neither commutative nor cocommutative (see \cite{kac1966finite} for a complete description).

As a $C^{*}$-algebra, $\mathbb{KP}=\mathbb{C}\oplus\mathbb{C}\oplus\mathbb{C}\oplus\mathbb{C}\oplus M_{2}(\mathbb{C})$.
Let $G=\mathbb{Z}_{2}\times\mathbb{Z}_{2}=\{e,a,b,c\mid a^{2}=b^{2}=c^{2}=e,\;ab=c\}$.
A typical element of $\mathbb{KP}$ is written $(\lambda_{e},\lambda_{a},\lambda_{b},\lambda_{c},A)$.
Denote by $\alpha_{e}=(1,0,0,0,0)$, $\alpha_{a}=(0,1,0,0,0)$, $\alpha_{b}=(0,0,1,0,0)$, $\alpha_{c}=(0,0,0,1,0)$,
and set $U_{e}=I_{2}$,
\[
U_{a}=\begin{bmatrix}0&i\\1&0\end{bmatrix},\qquad
U_{b}=\begin{bmatrix}0&1\\i&0\end{bmatrix},\qquad
U_{c}=\begin{bmatrix}-1&0\\0&1\end{bmatrix}.
\]
The Hopf $*$-algebra structure is given by
\begin{align*}
\Delta(\alpha_{y})&=\sum_{x\in G}\alpha_{x}\tens\alpha_{x^{-1}y}
                  +\frac12\bigl[(u_{y})_{im}(\bar{u}_{y})_{jn}\bigr]_{ijmn};\\
\Delta(A)&=\sum_{x\in G}\alpha_{x^{-1}}\tens U_{x}AU_{x}^{*}
          +\frac12\sum_{x\in G}\bar{U}_{x}AU_{x}^{T}\tens\alpha_{x};\\
\varepsilon(\lambda_{e},\lambda_{a},\lambda_{b},\lambda_{c},A)&=\lambda_{e};\\
S(\lambda_{e},\lambda_{a},\lambda_{b},\lambda_{c},A)&=(\lambda_{e},\lambda_{a},\lambda_{b},\lambda_{c},A^{T}),
\end{align*}
where $U_{y}=[(u_{y})_{ij}]$, $\bar{U}_{y}=[(\bar{u}_{y})_{ij}]$ and $^{T}$ denotes transpose.

The Haar state of $\mathbb{KP}$ is
\[
h=\frac18(\alpha_{e}^{*}+\alpha_{a}^{*}+\alpha_{b}^{*}+\alpha_{c}^{*})+\frac14(a_{11}^{*}+a_{22}^{*}),
\]
with $\alpha_{x}^{*},a_{ij}^{*}$ the linear functionals satisfying
\[
\alpha_{x}^{*}(\alpha_{y})=\delta_{xy,e},\quad \alpha_{x}^{*}(a_{jk})=0,\quad
a_{ij}^{*}(a_{kl})=\delta_{ik}\delta_{jl},\quad a_{ij}^{*}(\alpha_{x})=0,
\]
and $a_{ij}=0\oplus0\oplus0\oplus0\oplus E_{ij}$ ($E_{ij}$ are the matrix units).

The set $I(\mathbb{KP})$ of irreducible finite‑dimensional unitary representations consists of
\begin{align*}
\rho(1)&=1\oplus1\oplus1\oplus1\oplus\begin{bmatrix}1&0\\0&1\end{bmatrix},\;
\rho(2)=1\oplus-1\oplus-1\oplus1\oplus\begin{bmatrix}1&0\\0&-1\end{bmatrix},\\
\rho(3)&=1\oplus-1\oplus-1\oplus1\oplus\begin{bmatrix}-1&0\\0&1\end{bmatrix},\;
\rho(4)=1\oplus1\oplus1\oplus1\oplus\begin{bmatrix}-1&0\\0&-1\end{bmatrix},\\
\mathfrak{X}&=\begin{bmatrix}
1\oplus-1\oplus1\oplus-1\oplus\begin{bmatrix}0&0\\0&0\end{bmatrix}&
0\oplus0\oplus0\oplus0\oplus\begin{bmatrix}0&e^{i\pi/4}\\e^{-i\pi/4}&0\end{bmatrix}\\[4pt]
0\oplus0\oplus0\oplus0\oplus\begin{bmatrix}0&e^{-i\pi/4}\\e^{i\pi/4}&0\end{bmatrix}&
1\oplus1\oplus-1\oplus-1\oplus\begin{bmatrix}0&0\\0&0\end{bmatrix}
\end{bmatrix}.
\end{align*}
The dual quantum group $\widehat{\mathbb{KP}}$ is the finite quantum group
\[
\ell^{\infty}(\widehat{\mathbb{KP}})=\bigoplus_{\pi\in I(\mathbb{KP})} \Bop(H_{\pi})
=\mathbb{C}\oplus\mathbb{C}\oplus\mathbb{C}\oplus\mathbb{C}\oplus M_{2}(\mathbb{C}),
\]
with Haar weight $\widehat{h}=\alpha_{e}^{*}+\alpha_{a}^{*}+\alpha_{b}^{*}+\alpha_{c}^{*}+2(a_{11}^{*}+a_{22}^{*})$.

Now we compute the STFT on two specific pairs of elements and verify that the lower bound in
\eqref{eq:uncertainty-STFT} is attained.

1. Take $x_{1}=y_{1}=\alpha_{e}$.  Using \eqref{eq:STFT-algebraic}, we have
\[
\phi_{y_{1}}x_{1}=(\id\tens\cF)\bigl((1\tens\alpha_{e})(S\tens\id)\Delta(\alpha_{e})\bigr)
=\alpha_{e}\tens\cF\alpha_{e}.
\]
Since $\alpha_{e}$ is a projection and $8\cF\alpha_{e}=\alpha_{e}+\alpha_{a}+\alpha_{b}+\alpha_{c}+a_{11}+a_{22}$
is the range projection of $\cF\alpha_{e}$, we obtain
\[
\mathcal{S}(\phi_{y_{1}}x_{1})=h(\alpha_{e})\,\widehat{h}(8\cF\alpha_{e})=1.
\]

2. Take $x_{2}=a_{12},\;y_{2}=a_{22}$.  Again by \eqref{eq:STFT-algebraic}, we obtain
\[
\phi_{y_{2}}x_{2}=(\id\tens\cF)\bigl((1\tens a_{12})(S\tens\id)\Delta(a_{22})\bigr)
=(\alpha_{e}+\alpha_{c})\tens\cF a_{12}.
\]
Because $\cF a_{12}=\frac14 e^{-i\pi/4}a_{12}+\frac14 e^{i\pi/4}a_{21}$, the range projection of
$\phi_{y_{2}}x_{2}$ is $(\alpha_{e}+\alpha_{c})\tens(a_{11}+a_{22})$, whence
\[
\mathcal{S}(\phi_{y_{2}}x_{2})=1 .
\]
Thus, the equality in the Donoho--Stark uncertainty principle \eqref{eq:uncertainty-STFT} can indeed be achieved on a genuine non‑commutative, non‑cocommutative quantum group.
\end{example}
	
\bigskip

\bmhead{Acknowledgements} 
We are indebted to Hans G. Feichtinger for inspiring and stimulating communications following his lecture at the 2024 International Workshop on Frames, Operators and Related Topics at the Chern Institute of Mathematics.
We also thank Zhengwei Liu, Franz Luef, and Jinsong Wu for valuable communications following their lectures at Nankai University. 
Furthermore, we acknowledge the Institute for Advanced Study in Mathematics at the Harbin Institute of Technology for its hospitality during the summer workshops.

%The authors would also like to express their appreciation to the referees for carefully reading the manuscript and providing many helpful comments and suggestions that helped improve the representation of this paper.

\bmhead{Funding}
The research was supported by the National Natural Science Foundation of China (NSFC Grant Nos. 
11971348, 12071230 and 12471131).

\end{document}